\documentclass{amsart}
\usepackage[active]{srcltx}
\title[Isometric realization of cross caps]{%
  Isometric realization of
  cross caps as formal power series 
  and its applications
}
\date{January 23, 2016.}

\usepackage{graphicx} 
\numberwithin{equation}{section}
\theoremstyle{plain}
 \newtheorem{theorem}{Theorem}[section]
 \newtheorem*{theorem*}{Theorem}
 \newtheorem*{lemma*}{Lemma}
 \newtheorem{proposition}[theorem]{Proposition}
 \newtheorem{fact}[theorem]{Fact}
 \newtheorem*{fact*}{Fact}
 \newtheorem{lemma}[theorem]{Lemma}
 \newtheorem{corollary}[theorem]{Corollary}
 
 \newtheorem*{problem}{Problem}
 \newtheorem*{question}{Question}

\theoremstyle{remark}
 \newtheorem{definition}[theorem]{Definition}
 \newtheorem{remark}[theorem]{Remark}
 \newtheorem*{remark*}{Remark}
 \newtheorem*{ack}{Acknowledgements}
 \newtheorem{example}[theorem]{Example}

\numberwithin{equation}{section}
\renewcommand{\theenumi}{{\rm(\arabic{enumi})}}
\renewcommand{\labelenumi}{\theenumi}

\newcommand{\vect}[1]{\boldsymbol{#1}}
\newcommand{\Z}{\boldsymbol{Z}}
\newcommand{\R}{\boldsymbol{R}}
\newcommand{\A}{\mathcal{A}}
\newcommand{\B}{\mathcal{B}}
\newcommand{\E}{\mathcal{E}}
\newcommand{\F}{\mathcal{F}}
\newcommand{\G}{\mathcal{G}}
\newcommand{\Xx}{\mathcal{X}}
\newcommand{\Yy}{\mathcal{Y}}
\newcommand{\Zz}{\mathcal{Z}}
\newcommand{\W}{\mathcal{W}}

\renewcommand{\O}{\mathcal{O}}
\newcommand{\Hess}{\operatorname{Hess}}
\newcommand{\SO}{\operatorname{SO}}

\newcommand{\pmt}[1]{{\begin{pmatrix} #1  \end{pmatrix}}}
\newcommand{\innner}[2]{\langle{#1}|{#2}\rangle}

\renewcommand{\phi}{\varphi}
\renewcommand{\epsilon}{\varepsilon}

\thanks{%
  The first author was partially supported by the
  Grant-in-Aid for Challenging Exploratry Research, No.\ 
  26610016 of the Japan Society for the Promotions of Science (JSPS),
  the second author was partially supported by the Grant-Aid for JSPS
  fellows,
  No.\ 14J00101,
  the third author was partially 
  supported by the Grant-in-Aid for 
  Scientific Research (A) No.\ 262457005, 
  and
  the fourth author by (C) No.\ 26400087 from 
  JSPS
}%

\author{A.~Honda}
\address[Atsufumi Honda]{
  National Institute of Technology, 
  Miyakonojo College,
  473-1,
  Yoshiocho, Miyakonojo, 
  Miyazaki 885-8567,
  Japan
}
\email{atsufumi@cc.miyakonojo-nct.ac.jp}

\author{K.~Naokawa}
\address[Kosuke Naokawa]{%
   Department of Mathematics, Faculty of Science,
   Kobe University,
   Rokko, Kobe 657-8501, Japan
}
\email{naokawa@port.kobe-u.ac.jp}

\author{M.~Umehara}
\address[Masaaki Umehara]{%
   Department of Mathematical and Computing Sciences,
   Tokyo Institute of Technology
   2-12-1-W8-34, O-okayama, Meguro-ku,
   Tokyo 152-8552, Japan
}
\email{umehara@is.titech.ac.jp}

\author{K.~Yamada}
\address[Kotaro Yamada]{%
   Department of Mathematics,
   Tokyo Institute of Technology,
   O-okayama, Meguro, Tokyo 152-8551,
   Japan
}
\email{kotaro@math.titech.ac.jp}
\begin{document}
\begin{abstract}
 Two cross caps in Euclidean $3$-space
 are said to be \emph{formally isometric} if 
 their Taylor expansions of
 the first fundamental forms coincide
 by taking a suitable local coordinate system.
 For a given $C^\infty$ cross cap $f$,
 we give a method to find
 all cross caps which are 
 formally isometric to $f$.
 As an application, 
 we give a countable family of intrinsic invariants of cross caps
 which recognizes formal isometry classes completely.
\end{abstract}
\maketitle

\section*{Introduction}
Singular points of a positive semi-definite 
metric $d\sigma^2$ are the points where
the metric is not positive definite.
In the authors' previous work
\cite{HHNSUY} with Hasegawa and Saji, a class of 
positive semi-definite metrics on $2$-manifolds
called \lq Whitney metrics\rq\ 
was given.
Singularities of Whitney metrics
are isolated and the pull-back metrics
of cross caps in Euclidean $3$-space $\R^3$ are
typical examples of Whitney metrics. 

In \cite{HHNUY} with Hasegawa, the authors 
gave three intrinsic invariants
$\alpha_{2,0}$, $\alpha_{1,1}$ and $\alpha_{0,2}$ for cross caps.
After that they were generalized in \cite{HHNSUY}
as invariants of Whitney metrics.
In this paper, we construct a series of invariants
$\{\alpha_{i,j}\}_{i+j\ge 2}$ as an extension of
$\alpha_{2,0}, \alpha_{1,1}$ and $\alpha_{0,2}$.
This series of invariants
can distinguish isometric classes of real analytic 
Whitney metric completely (see Section~\ref{sec:5}),
and are related to the following problem:
\begin{problem}
 Can each  singular point of a Whitney metric locally 
 be isometrically realized as a cross cap in $\R^3$?
\end{problem}
The authors expect the answer will be
affirmative, under the assumption that
 the metric is real analytic.
In fact, for real analytic cuspidal 
edges and swallowtails,
the corresponding problems are solved affirmatively
by Kossowski \cite{K} (see also \cite{HHNSUY}).
Moreover, the moduli of isometric deformations of a 
given generic real analytic germ of cuspidal edge singularity was
completely determined in \cite{NUY}.
In this paper, 
we construct all isometric realizations of 
a given Whitney metric germ at their singularities 
as formal power series
solutions of the problem.
The above family of invariants
$\{\alpha_{i,j}\}_{i+j\ge 2}$
corresponds to the coefficients of
the Taylor expansion of
a certain realization of
the Whitney metric associated to a given cross cap 
singular point. So we can give an explicit algorithm
to compute the invariants (cf.\ Section \ref{sec:5}).
Although it seems difficult to show the
convergence of the power series,
we can
approximate it by $C^\infty$ maps
by applying Borel's theorem (cf.\ \cite[Lemma 2.5 in Chapter IV]{GG}),
and get our main result 
(cf.\ Theorem \ref{thm:main}).

\section{Preliminaries and main results}\label{sec:prelim}
\subsection{Characteristic functions of cross caps}
We recall fundamental properties of
cross caps (cf.\ \cite{W,FH,HHNUY}).
Let $f:U\to \R^3$ be a $C^\infty$ map,
where $U$ is a domain in $\R^2$.
A point $p$ $(\in U)$ is called a \emph{singular point}
if $f$ is not an immersion at $p$.
Consider such a map given by
\begin{equation}\label{eq:01}
   f_{0}(u,v)=(u,uv,v^2),
\end{equation}
which has an isolated singular point at the origin $(0,0)$
and is called the \emph{standard cross cap}.
A singular point $p$ of the map $f:U\to \R^3$
is called a \emph{cross cap} or a \emph{Whitney umbrella}
if there exist a local diffeomorphism $\phi$ on $\R^2$
and a local diffeomorphism $\Phi$ on $\R^3$ satisfying
$\Phi\circ f=f_{0}\circ \phi$ such that
$\phi(p)=(0,0)$ and $\Phi(f(p))=(0,0,0)$.

Let $f : (U; u, v) \to \R^3$ be a $C^\infty$ map 
such that $ (u, v) = (0, 0)$ is 
a cross cap singularity and $f_v(0, 0) = 0$.
Since cross cap singularities are co-rank one,
$f_u(0,0)\ne 0$.
We call the line
\[
   \{ f (0, 0) + t f_u(0, 0) ; t \in \R\}
\]
the \emph{tangential line} at the cross cap.
The plane passing through $f (0, 0)$ spanned
by $f_u(0, 0)$ and $f_{vv}(0, 0)$ is called the 
\emph{principal plane}. 
The principal plane is determined independently
of the choice of the local coordinate 
system $(u,v)$ satisfying $f_v(0, 0) = 0$.
By definition, the principal plane contains
the tangential line.

On the other hand, the plane passing through 
$f (0, 0)$ perpendicular to the tangential 
line is called the \emph{normal plane}. 
The unit normal vector $\nu(u, v)$ near 
the cross cap at $(u, v) = (0, 0)$ can
be extended as a $C^\infty$ function of $(r, \theta)$ 
by setting $u = r \cos \theta$ and $v = r \sin \theta$, 
and the limiting normal vector
\[
   \nu(\theta) := \lim_{r \to 0}
   \nu(r \cos \theta, r \sin \theta) \in
   T_{f (0,0)}\R^3
\]
lies in the normal plane. 

We have the following normal form of $f$ at a cross cap 
singularity:
\begin{fact}\label{fact:W}
 Let $f:(U;u,v)\to \R^3$ be a
 germ of a cross cap singularity at $(u,v)=(0,0)$.
 Then there exist an orientation
 preserving isometry $T$ and a
 local diffeomorphism
 $(x,y)\mapsto (u(x,y),v(x,y))$ such that
 $f(x,y):=f(u(x,y),v(x,y))$
 satisfies
 \begin{equation}\label{eq:cross}
    T\circ f(x,y)=
        \left(
	   x,\,\,
	   xy+b(y),\,\,
	   z(x,y)\right),
 \end{equation}
 where  $b(y)$ and $z(x,y)$ are smooth functions
 satisfying 
 \begin{equation}\label{eq:cross2}
  b(0)=b'(0)=b''(0)=0,\quad
   z(0,0)=z_x(0,0)=z_y(0,0)=0,\quad z_{yy}(0,0)>0.
 \end{equation}
 Moreover, if we assume
 \begin{equation}\label{eq:J}
  \det\pmt{x_u & x_v \\ y_u & y_v}>0
 \end{equation}
 at $(u,v)=(0,0)$, then the function germs 
 $x=x(u,v), y=y(u,v)$, 
 $z=z(x,y)$ and $b=b(y)$ 
 are uniquely determined.
\end{fact}

This special local coordinate system $(x,y)$ 
is called the \emph{canonical coordinate system}
or the \emph{normal form} of $f$ at the cross cap singularity.
In particular, the function $b(y)$ is called 
the \emph{characteristic function}
associated to the cross cap $f$.
Historically, West \cite{W} initially
introduced this normal form of cross caps.
Unfortunately, it is difficult to have
an access to the reference \cite{W}.  
An algorithm approach to
determine the coefficients
of the Taylor expansions of $b(y)$ 
and $z(x,y)$
can be found in Fukui-Hasegawa
 \cite[Proposition 2.1]{FH}, which we will
apply at Section~\ref{sec:5}.
For the sake of the later discussions,
we give here a proof of the last
assertion of Fact \ref{fact:W}
as follows.
\begin{proof}[Proof of the uniqueness of the normal form]
 Without loss of generality, we may assume that
 $f(0,0)=(0,0,0)$. 
 Suppose that there exists another such
 normal form
 \begin{equation}\label{eq:tilde}
  \tilde T \circ \tilde f(\tilde x,\tilde y)=
          \bigl(\tilde x,\tilde x\tilde y+\tilde b(\tilde y),
	  \tilde z(\tilde x,\tilde y)\bigr),
 \end{equation}
 where 
 $\tilde f(\tilde x,\tilde y):=f(u(\tilde x,\tilde y),v(\tilde x,\tilde y))$.
 Since $f(0,0)=(0,0,0)$, two isometries
 $T$ and $\tilde T$ can be considered as
 matrices in $\SO(3)$.
 By \eqref{eq:cross} and \eqref{eq:tilde}, it holds that
 $T(f_x (0,0)) =\tilde T(\tilde f_{\tilde x} (0,0))= \vect e_1$,
 where $\vect e_1:=(1,0,0)$.
 Since the tangential lines of $f$ and $\tilde f$ coincide, we have
 \[
   \tilde T \circ T^{-1}(\R \vect e_1) = \tilde T(\R f_x(0,0)) = \tilde
   T(\R \tilde f_{\tilde x}(0,0)) = \R \vect e_1.
 \]
 Hence, $\vect e_1$ is an eigenvector of the matrix
 $S:=\tilde T\circ T^{-1}$.
 On the other hand, by \eqref{eq:cross} and \eqref{eq:tilde} 
 again, both of
 $T(f_{yy}(0,0))$ and $\tilde T(\tilde{f}_{\tilde{y} \tilde{y}}(0,0))$
 must be proportional to $\vect e_3:=(0,0,1)$.
 Since the principal planes of $f$ and $\tilde f$ coincide, we have
 \begin{align*}
   \tilde T \circ T^{-1}(\R \vect e_1 + \R \vect e_3)
    &= \tilde T(\R f_x(0,0) + \R f_{yy}(0,0)) \\
    & = \tilde T(\R \tilde{f}_{\tilde x}(0,0) + \R \tilde{f}_{\tilde{y}
  \tilde{y}}(0,0))
    = \R \vect e_1 + \R \vect e_3.
 \end{align*}
 Since we know that $\vect e_1$ is
 an eigenvector of $S$, we can
 conclude that $\vect e_3$ is also
 an eigenvector of $S$.
 Thus $\vect e_2=(0,1,0)$ is also
 an eigenvector of $S$, and we can write
 \[
   S=
    \pmt{
      \epsilon_1 & 0 & 0 \\
      0 & \epsilon_2 & 0 \\
      0 & 0 & \epsilon_1\epsilon_2
    }
     \qquad (\epsilon_i=\pm 1,~i=1,2).
 \]
 Then we get the expression
 \[
   \pmt{
        \epsilon_1 & 0 & 0 \\
        0 & \epsilon_2 & 0 \\
        0 & 0 & \epsilon_1\epsilon_2
   }
   \pmt{
     x\\xy+b(y)\\
     z(x,y)
   }
   = 
   \pmt{
       \tilde x\\ 
       \tilde x\tilde y+\tilde b(\tilde y)\\
       \tilde z(\tilde x,\tilde y)
    }.
 \]
 Comparing the first components, we have
 \begin{equation}\label{eq:x}
  \epsilon_1 x=\tilde x.
 \end{equation}
 Next, comparing the second components, 
 we have
 \begin{equation}\label{eq:middle}
  \epsilon_2(xy+b(y))
   =
   \epsilon_1 x \tilde y+\tilde b(\tilde y).
 \end{equation}
 Substituting $x=0$,
 we get 
 $\epsilon_2b(y)=\epsilon_1 \tilde b(\tilde y)$,
 and therefore
 $\epsilon_2 x y=\epsilon_1 x\tilde y$.
 So we can conclude that
 $\tilde y=\epsilon_1\epsilon_2 y$.
 By \eqref{eq:J}, we have $\epsilon_2=1$.
 By comparing the third components, 
 $\epsilon_1 z(x,y)=\tilde z(\epsilon_1 x,\epsilon_1 y)$ 
 holds.
 Hence we have
 \[
    \epsilon_1 z_{yy}(x,y)=
     \tilde z_{yy}(\epsilon_1 x,\epsilon_1 y)
     =\tilde z_{\tilde y\tilde y}(\tilde x,\tilde y)
     >0.
 \]
 Since $z_{yy}(x,y)>0$, we can conclude that $\epsilon_1=1$.
 In particular, we have $x=\tilde x$, $y=\tilde y$,
 and $z(x,y)$ coincides with $\tilde z(\tilde x, \tilde y)$.
 Then \eqref{eq:middle} reduces to
 $b(y)=\tilde b(y)$, proving the assertion.
\end{proof}

In the statement of
Fact \ref{fact:W},
$b$ and $z$ can be
taken as  real analytic functions
if  $f$ is real analytic. 
The following assertion was proved in \cite{HHNUY}:
\begin{fact}
 The characteristic function 
 $b(y)$ vanishes identically if and only if
 the set of self-intersections of $f$
 lies in the intersection of 
 the principal plane and the normal plane.
\end{fact}
\begin{definition}\label{def:normal}
 Cross caps whose characteristic functions
 vanish identically
 are called \emph{normal cross caps}
 (cf.\ \cite{HHNUY}).
\end{definition}

Let $C^\infty_o(\R^2)$ (resp.\ $C^\infty_o(\R)$)
be the set of  $C^\infty$ function germs at
the origin $o$ of the $(u,v)$-plane $\R^2$
(resp.\ the line $\R$).
Two functions
$h_1(u,v)$, $h_2(u,v)\in C^\infty_o(\R^2)$
(resp.\ $h_1(t),h_2(t)\in C^\infty_o(\R)$)
are called \emph{equivalent}
(denoted by $h_1\sim h_2$) 
if the Taylor series of $h_1$ coincides
with that of $h_2$ at the origin.
By the well-known Borel theorem (cf.\ \cite[Lemma 2.5 in Chapter IV]{GG}), 
the quotient space $C^\infty_o(\R^2)/\sim$
(resp.\ $C^\infty_o(\R)/\sim$)
can be identified with the 
space $\R[[u,v]]$ (resp.\ $\R[[t]]$)
of formal power series in the variables $u,v$
(resp.\ $t$) at the origin $o$,
that is, the formal power series
\begin{equation}\label{eq:h}
   [h]:=\sum_{k,l=0}^\infty 
          \frac{\partial^{k+l} h(0,0)}
	  {\partial u^k\partial v^l}\frac{u^kv^l}{k!l!}
	  \qquad
	  \left(\mbox{resp.}~ 
	    [h]:=
	    \sum_{j=0}^\infty \frac{d^j h(0)}{dt^j}\frac{t^j}{j!}
	  \right)
\end{equation}
represents the equivalent class containing $h$
in $C^\infty_o(\R^2)/\sim$ (resp.\ in $C^\infty_o(\R)/\sim$).
At the end of this subsection,
we show that the following assertion
is an immediate consequence of our main
result (Theorem \ref{thm:main}): 
\begin{proposition}\label{prop:main}
 Let $f_j:(U;u,v)\to \R^3$ $(j=1,2)$
 be two real analytic 
 cross cap singularities such that the 
 first fundamental form
 {\rm(}i.e.\ the pull back of the
 canonical metric
 of $\R^3${\rm)} 
 of $f_1$ coincides with that of $f_2$.
 Then,  $f_1$ coincides with
 $f_2$ up to orientation-preserving
 isometries  in $\R^3$
 if and only if the Taylor series
 of their characteristic
 functions coincide.
\end{proposition}

\begin{figure}[htb]
 \centering
          \includegraphics[height=3.0cm]{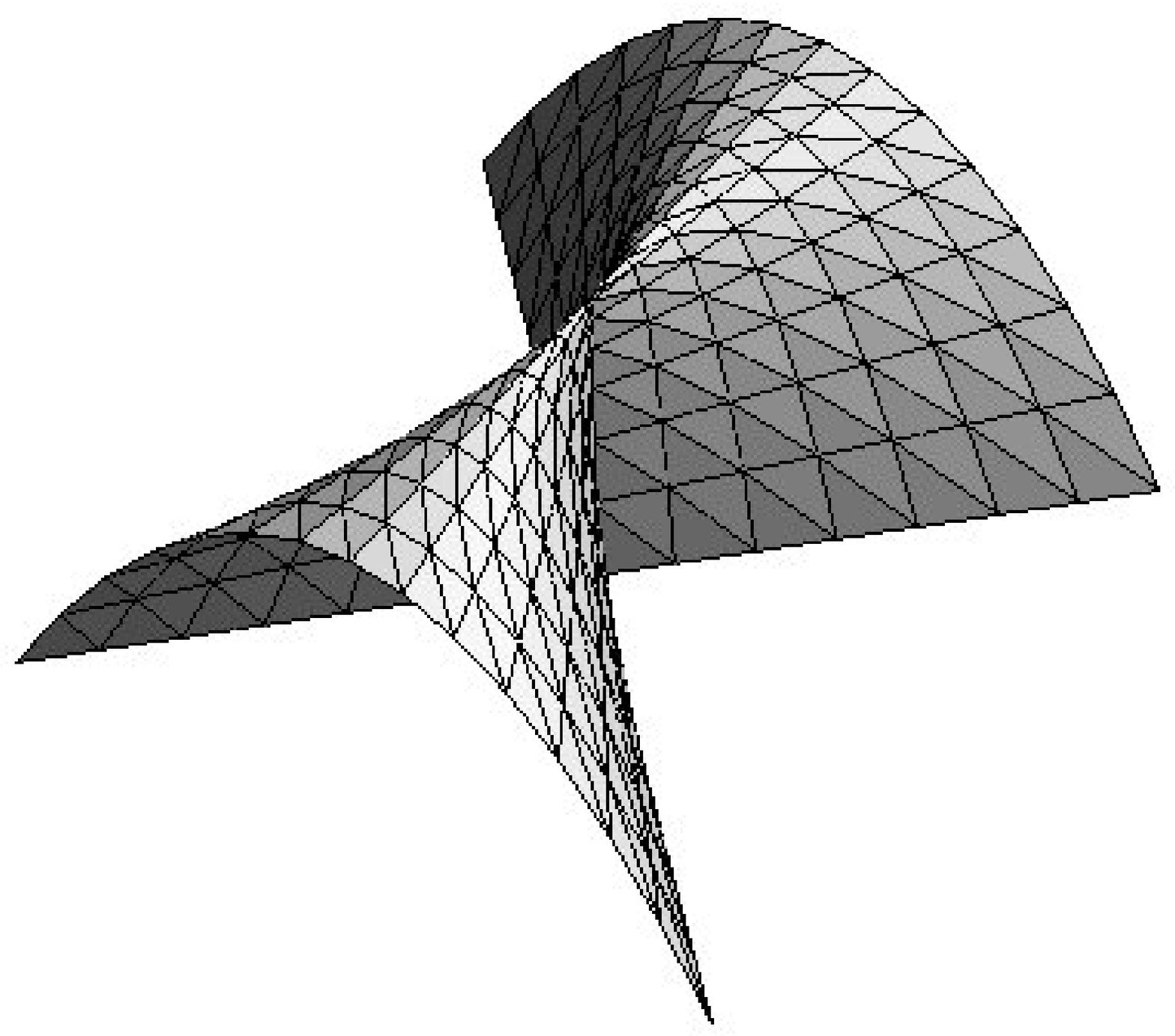}
   \quad  \includegraphics[height=3.0cm]{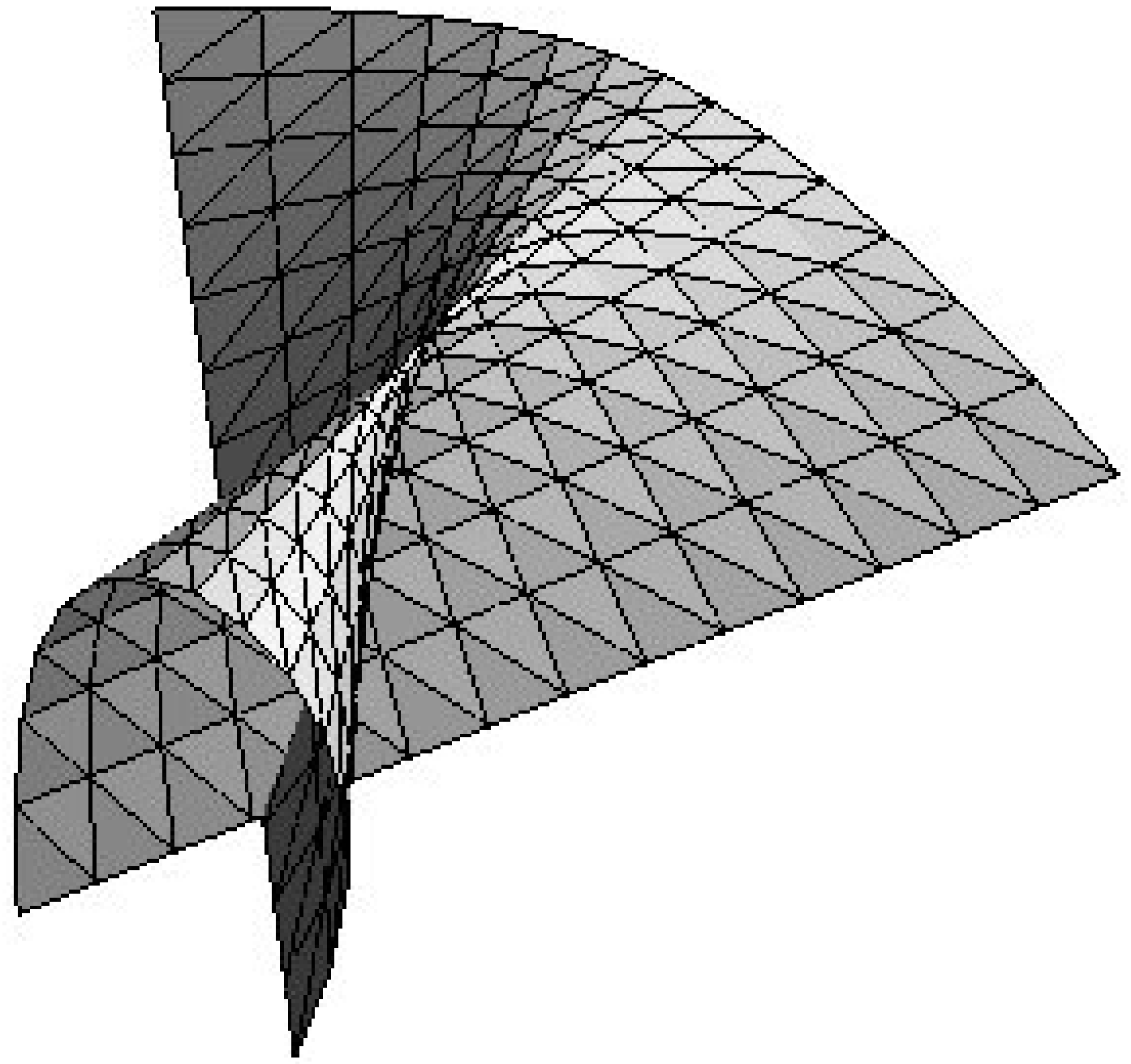}
   \quad  \includegraphics[height=3.0cm]{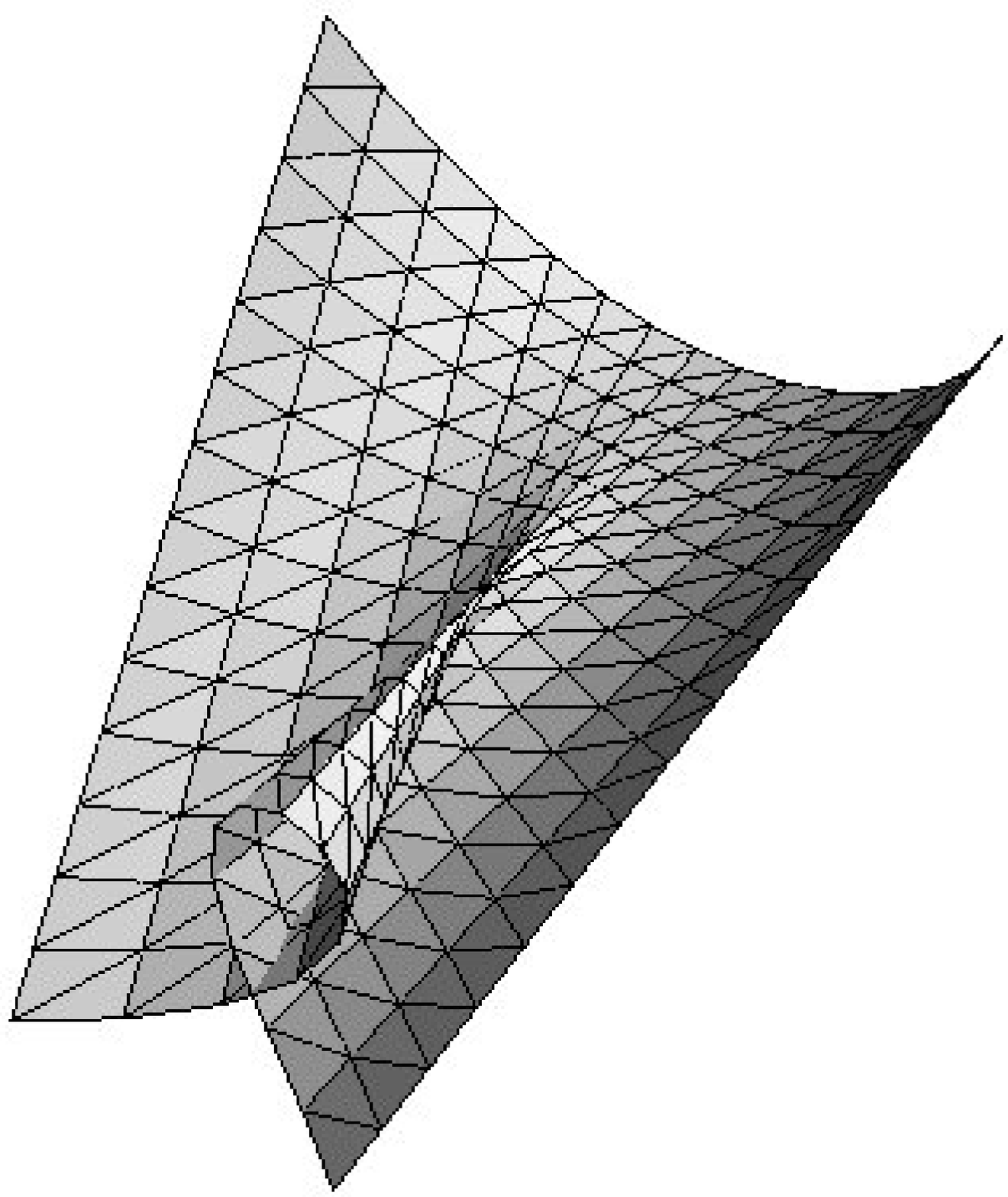}
  \caption{An isometric deformation of the standard cross cap.}
  \label{fig:deform}
\end{figure}

Proposition \ref{prop:main} tells us that an analytic isometric 
deformation of cross caps can be controlled 
by the corresponding deformation of characteristic
functions. 
Examples of isometric deformations 
of cross caps are constructed
in \cite{HHNUY} (cf.\ Figure \ref{fig:deform}).
By the definition of normal cross caps (cf.\ Definition 
\ref{def:normal}),
we get the following corollary:

\begin{corollary}[The rigidity of normal cross caps]
\label{cor:rigidity}
 Two germs of real analytic normal cross caps are congruent
 if and only if they have the same first fundamental form.
\end{corollary}

Corollary~\ref{cor:rigidity} 
suggests us the following:
\begin{question}
 Can a given cross cap germ in $\R^3$
 be isometrically deformed into a normal cross cap?
\end{question}

If the answer to the problem 
in the introduction is affirmative, 
so it is for the above question.
Since the standard cross cap (cf.\ \eqref{eq:01})
is normal, the deformation of the standard 
cross cap
in Figure \ref{fig:deform}
can be re-interpreted as
a normalization of the rightmost
cross cap to the normal cross cap (i.e.\ the leftmost cross cap). 
We give here another example:

\begin{example}\label{ex:normalize}
 We set
 \[
    f_0(u,v)=\left(u,uv+\frac{v^3}6,\frac{u^2}2+\frac{v^2}2\right).
 \]
 This gives the normal form of a cross cap at $(0,0)$.
 (see Figure \ref{fig:example}, left).
 Since $b\ne 0$, this cross cap is not normal.
 We suppose that there exists a real analytic
 germ $f_1$ of a normal cross cap 
 which is isometric to $f_0$.
 By Corollary \ref{cor:rigidity},
 we know the uniqueness of $f_1$.
 Moreover, for a given positive integer $n$,
 we can determine the coefficients of
 its Taylor expansion of order at most $n$
 using our algorithm as in the proof of
 Theorem \ref{thm:main}.
 Figure \ref{fig:example}, right
 is an approximation of $f_1$ by setting $n=10$.
 The main difference between the 
 figures of $f_0$ and $f_1$ appears 
 on the  set of self-intersection.
 The set of self-intersection of the figure of $f_1$
 consists of a straight line 
 perpendicular to the tangential direction of
 the surface at $(0,0)$.
\end{example}

\begin{figure}[htb]
 \centering
  \includegraphics[height=2.0cm]{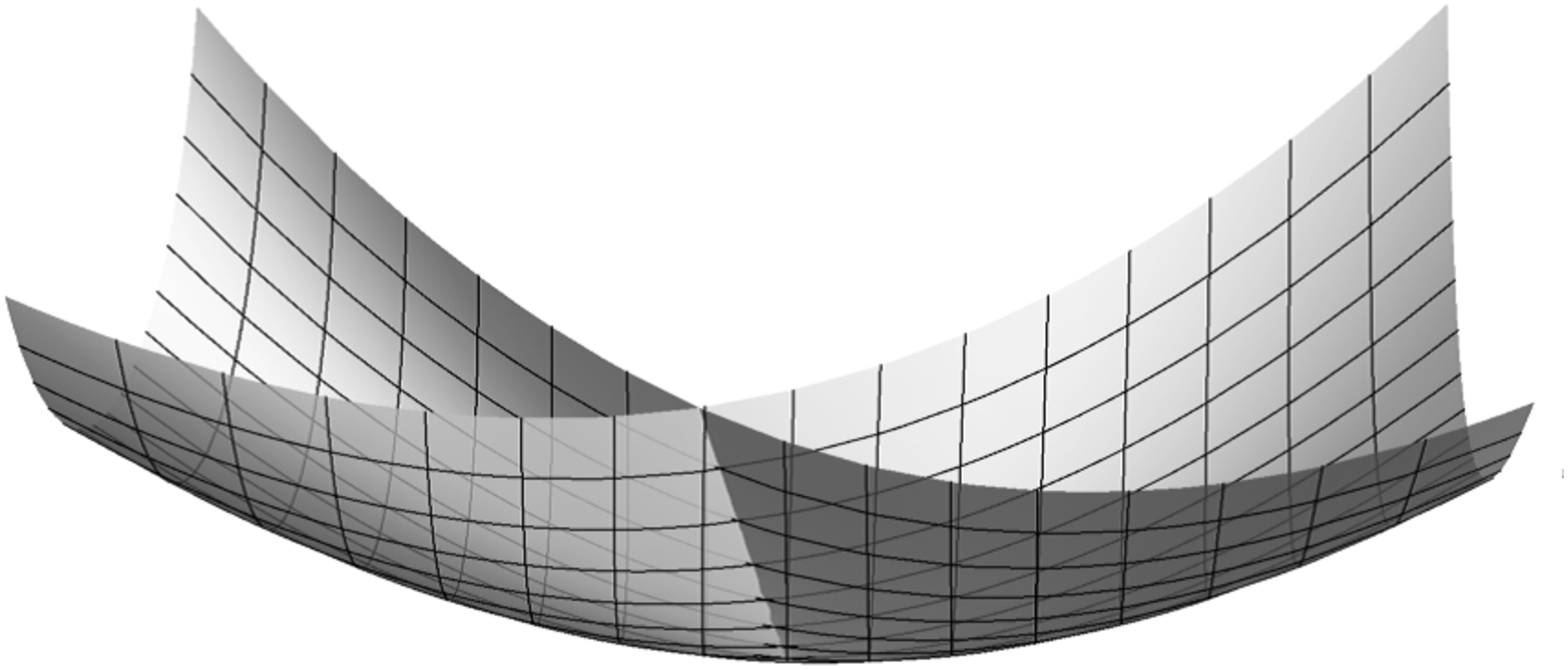}
  \quad  \includegraphics[height=2.0cm]{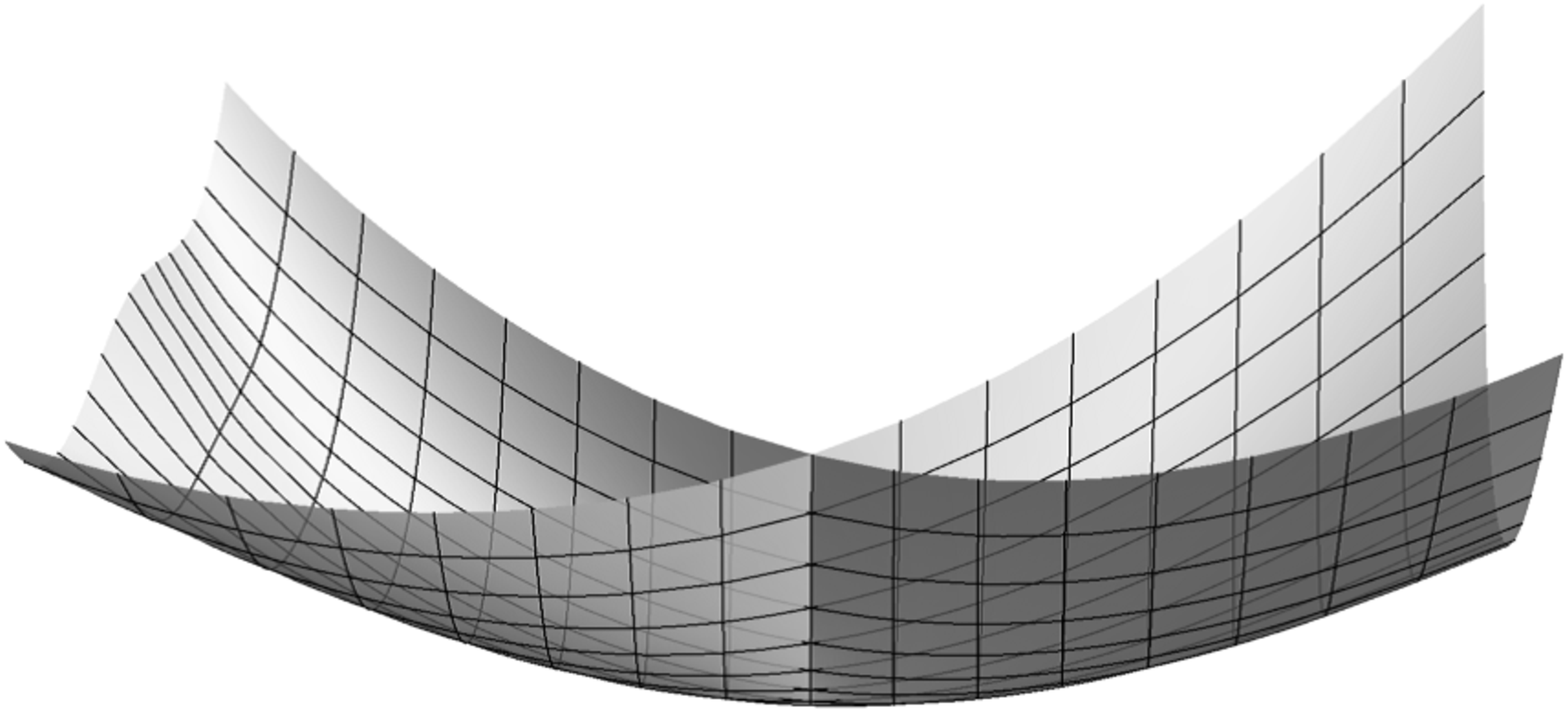}
  \caption{Example~\ref{ex:normalize}:
  The cross cap $f_0$ and its corresponding normal cross cap $f_1$.}
 \label{fig:example}
\end{figure}

\subsection{Whitney metrics}
We fix  a $2$-manifold $M^2$,
and a positive semi-definite metric $d\sigma^2$
on $M^2$.  A point $p\in M^2$ is called a 
\emph{singular point} of 
the metric $d\sigma^2$ if the metric 
is not positive definite at $p$.

Let $p$ be a singular point of
$d\sigma^2$, and $(u,v)$ a local coordinate system
centered at  $p$.
We set 
\begin{equation}\label{eq:efg}
 d\sigma^2=E\,du^2+2F\,du\,dv+G\,dv^2.
\end{equation}
The local coordinate system $(u,v)$ is called \emph{admissible}
if $\partial/\partial v$ is a null direction of the
metric $d\sigma^2$, that is,
it holds that $F=G=0$ at the origin.

\begin{definition}\label{def:adms}
 A singular point $p$ of the metric $d\sigma^2$
 is called 
 \emph{admissible}\footnote{
 Admissibility was originally introduced by 
 Kossowski \cite{K}. 
 He called it
 $d(\langle,\rangle)$-\emph{flatness}.
 Our definition of admissibility is equivalent to
 the original one, see \cite[Proposition 2.7]{HHNSUY}.
}%
 if there exists an admissible local coordinate system $(u,v)$ 
 centered at $p$ satisfying
 \[
    E_v=2F_u,\qquad G_u=G_v=0
 \]
 at the origin.  
 If each singular point of $d\sigma^2$
 is admissible, then $d\sigma^2$ is called \emph{admissible}.
\end{definition}

\begin{definition}[\cite{HHNSUY}]\label{def:hessian}
 Let  $p$ be a singular point of 
 an admissible (positive semi-definite) metric 
 $d\sigma^2$
 on $M^2$ in the sense of Definition~\ref{def:adms}.
 Let $(u,v)$ be an admissible local coordinate system
 centered at $p$ 
 and set
 \[
    \delta:=EG-F^2,
 \]
 where $E,F,G$ are functions
 satisfying \eqref{eq:efg}.
 If the Hessian
 \[
    \Hess_{u,v}(\delta):=
    \det\pmt{\delta_{uu} & \delta_{uv}\\ \delta_{uv} & \delta_{vv}}
 \]
 does not vanish at $p$, then $p$ is called
 an \emph{intrinsic cross cap} of
 $d\sigma^2$.
 Moreover, if $d\sigma^2$ admits only intrinsic cross cap
 singularities on $M^2$, 
 then it is called a \emph{Whitney metric}
 on $M^2$.
\end{definition}

The definition of intrinsic cross caps
is independent of the choice of admissible coordinate systems.
A Gauss-Bonnet type formula for Whitney metrics
is given in \cite{HHNSUY}.

\begin{definition}\label{def:inf}
 Two metrics $d\sigma^2_j$ ($j=1,2$)
 defined on a neighborhood of $p\in M^2$
 are called \emph{formally isometric}
 at $p$ if there exists 
 a local coordinate system $(u,v)$
 centered at $p$ such that
 (see \eqref{eq:h}  for the definition of the bracket $[~]$)
 \[
    [E_1]=[E_2],\quad [F_1]=[F_2],\quad [G_1]=[G_2]
 \]
 hold at $(0,0)$, where
 \[
       d\sigma^2_j=E_j\,du^2+2F_j\,du\,dv+G_j\,dv^2
       \qquad (j=1,2).
 \]
 We write $d\sigma^2_1\approx d\sigma^2_2$
 if two metrics are formally isometric.
 Two cross cap germs are said to be 
 \emph{formally isometric} if their induced metrics are
 formally isometric.
\end{definition}

The following is the main result of this paper:

\begin{theorem}\label{thm:main}
 Let $p$ be a singular point of
 a Whitney metric $d\sigma^2$.
 For any choice of  
 $C^\infty$ function germ $b\in C^\infty_o(\R)$
 satisfying $b(0)=b'(0)=b''(0)=0$,
 there exists a $C^\infty$ map germ $f$
 into $\R^3$
 having a cross cap singularity
 at $p$ satisfying the following two properties:
 \begin{enumerate}
  \item the first fundamental form of $f$
	{\rm(}i.e.\ the pull-back of the canonical metric of $\R^3$
	by $f${\rm)} is formally isometric to $d\sigma^2$
	at $p$, 
  \item the characteristic function of $f$ 
	is equivalent to  $b$,
	that is, it has the same Taylor expansion
	at $0$ as $b$.
 \end{enumerate}
 Moreover, such an $f$ is uniquely determined
 up to addition of flat functions%
\footnote{
 A $C^\infty$ function $h(u,v)$ is called \emph{flat} 
 (at $p$)
 if
 $\partial^{k+l} h(p)/\partial u^k\partial v^l$
 vanishes at $p$ for all non negative integers $k$, $l$.
}%
 at $p$.
 In other words, the Taylor expansion of $f$ 
 gives a unique formal power series solution 
 for the realization problem of the Whitney metric
 $d\sigma^2$ as a cross cap.
\end{theorem}

If the problem in the introduction
is affirmative, then the set of 
analytic cross cap germs
which have the same first fundamental form
can be identified with the set of
convergent power series in one variable.

\begin{proof}[Proof of Proposition \ref{prop:main}]
 The uniqueness of $f$ modulo flat functions
 and the second assertion of 
 Theorem \ref{thm:main} immediately
 imply Proposition \ref{prop:main},
 by setting $d\sigma^2=f_1^*ds^2_{\R^3}$
 and $f=f_2$, where $ds^2_{\R^3}$ is the canonical metric
 of the Euclidean $3$-space $\R^3$.
\end{proof}

\subsection{The strategy of the proof of Theorem \ref{thm:main}}
From now on, we fix a Whitney metric
\[
  d\sigma^2=E\,du^2+2F\,du\,dv+G\,dv^2
\]
defined on a neighborhood $U$ of 
the origin $o=(0,0)$ in the $(u,v)$-plane 
$\R^2$. 
We suppose that $o$ is a
singular point of $d\sigma^2$.
We set
\[
  \E:=[E],\qquad \F:=[F],\qquad \G:=[G],
\]
that is, $\E$, $\F$, $\G$ are the formal power series
in $\R[[u,v]]$ associated to the coefficients of
the metric $d\sigma^2$.
\begin{definition}\label{def:order}
 A formal power series 
 \begin{equation}\label{eq:P}
  P:=\sum_{k,l=0}^{\infty}
       \frac{P(k,l)}{k!l!}u^k v^l\qquad (P(k,l)\in\R)
 \end{equation}
 in $\R[[u,v]]$
 is said to be of \emph{order at least $m$} if
 \[
     P(k,l)=0 \qquad (k+l< m).
 \]
\end{definition}

We denote by $\O_m$ the ideal of $\R[[u,v]]$
consisting of series of order at least $m$.
By definition, $\O_0=\R[[u,v]]$.
In \cite{HHNSUY},  the following assertion
was given:
\begin{fact}[{\cite[Theorem 4.11]{HHNSUY}}]\label{fact:hhnsuy}
 One can choose the
 local coordinate system $(u,v)$ centered at the 
 singular point of $d\sigma^2$
 so that 
 \begin{alignat}{3}
  \label{eq:comp-E}
  \E\ (=[E]) &= 1 + 
       &&a_{2,0}^2 u^2 + 2 a_{2,0}a_{1,1}u v + (1+a_{1,1}^2)v^2
       &&+O_3(u,v),\\
  \label{eq:comp-F}
  \F\ (=[F]) &=     &&a_{2,0}a_{1,1} u^2 
           + (a_{2,0}a_{0,2}+a_{1,1}^2+1)u v + a_{1,1}a_{0,2}v^2
       &&+O_3(u,v),\\
  \label{eq:comp-G}
  \G\ (=[G]) &=     &&(1+a_{1,1}^2) u^2 + 
           2 a_{1,1}a_{0,2}u v + a_{0,2}^2v^2
       &&+O_3(u,v),
 \end{alignat}
 where $a_{2,0}$, $a_{1,1}$ and $a_{0,2}\ (>0)$
 are real numbers%
\footnote{
 As shown in \cite{HHNSUY}, $a_{2,0}$, $a_{1,1}$ and $a_{0,2}$
 are invariants of the Whitney metric $d\sigma^2$.
}, 
 and $O_m(u,v)$ $(m=1,2,3,\dots)$
 is a certain element of $\O_m$.
\end{fact}

So we can assume that our local coordinate system $(u,v)$ satisfies
\eqref{eq:comp-E},  \eqref{eq:comp-F} and  \eqref{eq:comp-G}.
We set
\begin{align*}
 \E=\sum_{k,l=0}^{\infty}
       \frac{\E(k,l)}{k!l!}u^k v^l\qquad
           (\E(k,l)\in\R), \\
 \F=\sum_{k,l=0}^{\infty}
       \frac{\F(k,l)}{k!l!}u^k v^l\qquad 
           (\F(k,l)\in\R), \\
 \G=\sum_{k,l=0}^{\infty}
       \frac{\G(k,l)}{k!l!}u^k v^l\qquad 
           (\G(k,l)\in\R). 
\end{align*}
We now fix a $C^\infty$ function germ $b(t)$
satisfying $b(0)=b'(0)=b''(0)=0$.
\begin{lemma}\label{lem:realize-zero}
 Let 
 \[
    f=\biggl(x(u,v),\,\,x(u,v)y(u,v)+b\bigl(y(u,v)\bigr),\,\,z(u,v)\biggr)
 \]
 be a $C^\infty$ map having cross cap singularity
 at $(0,0)$ satisfying
 \eqref{eq:J},
 $x(0,0)=y(0,0)=0$ and
 \begin{equation}\label{eq:bz-init}
  b(0)=b'(0)=b''(0)=0,\quad
     z(0,0)=z_u(0,0)=z_v(0,0),\quad z_{vv}(0,0)>0
 \end{equation}
 such that the first fundamental form
 of $f$ coincides with the Whitney metric $d\sigma^2$.
 Moreover, suppose that $(u,v)$ satisfies
 \eqref{eq:comp-E},
 \eqref{eq:comp-F} and  \eqref{eq:comp-G}.
 Then, 
 \[
    x_u(0,0)= \pm 1, \quad
    x_v(0,0)=0,\quad
    x_u(0,0)y_v(0,0)>0.
 \]
\end{lemma}
\begin{proof}
 Since $x(0,0)=y(0,0)=0$, $b'(0)=0$, 
 and $z_u(0,0)=0$, we have
 $f_u(0,0)=\bigl(x_u(0,0),0,0\bigr)$.
 In particular,
 \[
    1=\E(0,0) = 
      f_u(0,0)\cdot f_u(0,0)= x_u(0,0)^2
 \]
 holds and $x_u(0,0)=\pm 1$.
 On the other hand, we have
 \[
    0 = \F(0,0)=f_u(0,0)\cdot f_v(0,0)= x_u(0,0)x_v(0,0),
 \]
 and we get $x_v(0,0)=0$.
 By \eqref{eq:J},
 \[
   0<\left.\frac{\partial (x,y)}{\partial(u,v)}
                  \right|_{(u,v)= (0,0)}
        = x_u(0,0)y_v(0,0)
 \]
 holds, proving the assertion.
\end{proof}

Replacing $(u,v)$ by $(-u,-v)$ if necessary, 
we may assume that
\begin{equation}\label{eq:2nd-N}
  x_u(0,0) =1,\qquad y_v(0,0)>0.
\end{equation}
The map $f$ as in Lemma \ref{lem:realize-zero}
satisfies
 \begin{align}
  f_u\cdot f_u&= (1+y^2)x_u^2 + 2 \bigl(x+b'(y)\bigr)yx_uy_u
                 + \bigl(x^2+2xb'(y)+b'(y)^2\bigr)y_u^2+z_u^2,
  \label{eq:E}\\
  f_u\cdot f_v&= (1+y^2)x_ux_v + \bigl(x+b'(y)\bigr)y
                                (x_uy_v+x_vy_u) 
  \label{eq:F}\\
  \nonumber
  & \phantom{(1+y^2)x_ux_v} 
             + \bigl(x^2+2xb'(y)+b'(y)^2\bigr)y_uy_v+z_uz_v,
  \\
  f_v\cdot f_v&= (1+y^2)x_v^2 + 2\bigl(x+b'(y)\bigr)y x_vy_v
           + \bigl(x^2+2xb'(y)+b'(y)^2\bigr)y_v^2+z_v^2,
  \label{eq:G}
\end{align}
where
\begin{equation}\label{eq:beta}
 b'(t):=\frac{db(t)}{dt}.
\end{equation}

\begin{definition}\label{def:term}
 We call 
 \begin{equation}\label{eq:term}
  x_u^2, \quad y^2x_u^2,\quad
   xyx_uy_u,\quad b'(y)yx_uy_u,\quad
   x^2y_u^2,\quad  xb'(y)y_u^2, \quad b'(y)^2y_u^2,\quad 
   z_u^2
 \end{equation}
 the \emph{terms of $f_u\cdot f_u$}.
 Similarly,
 the \emph{terms of $f_u\cdot f_v$}
 (resp.\ the \emph{terms of $f_v\cdot f_v$})
 are also defined.
\end{definition}

We consider the following three polynomials in $u,v$:
\begin{align}
 \label{eq:X00}
 \Xx_{m+1}&:=u+\sum_{2\le k+l\le m+1} \frac{X(k,l)}{k!l!}u^ku^l, \\
 \label{eq:Y00}
 \Yy_{m-1}&:=\sum_{1\le k+l\le m-1} \frac{Y(k,l)}{k!l!}u^kv^l, \\
 \label{eq:Z00}
 \Zz_m&:=\sum_{2\le k+l\le m} \frac{Z(k,l)}{k!l!}u^kv^l.
\end{align}
We set
\[
   f^m:=\bigl(
             \Xx_{m+1},
	     \Xx_{m+1}\Yy_{m-1}+\beta_{m+1}(\Yy_{m-1}),
   	     \Zz_{m}
	\bigr),\qquad
   	\beta_{m+1}(t):=\sum_{j=3}^{m+1}\frac{b_j}{j!}t^j,
\]
where 
$[b]=\sum_{j=3}^{\infty}{b_j}t^j/{j!}$.
By definition, each coefficient of $f^m$ is a polynomial 
in the coefficients of $\Xx_{m+1}$, $\Yy_{m-1}$
and $\Zz_{m}$.
To describe the key assertion (cf.\ Proposition \ref{prop:key}), 
we prepare a terminology as follows:
\begin{definition}\label{def:terms}
 A triple of polynomials  $(\Xx_{m+1},\Yy_{m-1},\Zz_{m})$
 as in \eqref{eq:X00}, \eqref{eq:Y00} and \eqref{eq:Z00}
 are called the \emph{$m$-th formal solution} if they satisfy
 (cf.\ \eqref{eq:bz-init} and \eqref{eq:2nd-N})
\begin{align*}
 &X(0,0)=Y(0,0)=0,\quad X(1,0)=1,\quad Y(0,1)>0, \\
 &Z(0,0)=Z(1,0)=Z(0,1)=0,\quad Z(0,2)>0
\end{align*}
and
\begin{align}
\label{eq:E0}
 \E&=[f^m_u\cdot f^m_u]+O_{m+1}(u,v), \\
 \label{eq:F0}
 \F&=[f^m_u\cdot f^m_v]+O_{m+1}(u,v), \\
 \label{eq:G0}
 \G&=[f^m_v\cdot f^m_v]+O_{m+1}(u,v), 
\end{align}
where $O_{m+1}(u,v)$ is a term belonging to $\O_{m+1}$.
\end{definition}

The key assertion, which we would like to prove in Section~\ref{sec:App},
is stated as follows:

\begin{proposition}\label{prop:key}
 Let $(u,v)$ be a local coordinate system
 satisfying  \eqref{eq:comp-E},
 \eqref{eq:comp-F} and  \eqref{eq:comp-G}.
 Then for each $m\ge 2$, there exists
 a unique $m$-th formal solution.
\end{proposition}
We prove here the case $m=2$ of the proposition:
\begin{lemma}\label{lem:2}
 Let $(u,v)$ be a local coordinate system
 satisfying  \eqref{eq:comp-E},
 \eqref{eq:comp-F} and  \eqref{eq:comp-G}.
 Then there exists
 a unique second formal solution.
 More precisely it has the following expressions:
 \begin{align*}
  \Xx_{3}&=u,\qquad
  \Yy_1=v,\\
  \Zz_2&=\frac12 (a_{2,0}u^2+2 a_{1,1}uv +a_{0,2}v^2).
 \end{align*}
\end{lemma}
\begin{proof}
 By a straightforward calculation
 using $b(0)=b'(0)=b''(0)=0$, 
 we have
 \[
    [f_u\cdot f_u]=
     1+2 u X(2,0)+2 v X(1,1)+O_2(u,v).
 \]
 Since 
 $1=\E= [f_u\cdot f_u]+O_2(u,v)$,
 we can conclude that
 $X(2,0)=X(1,1)=0$.
 Similarly, using $z_u(0,0)=z_v(0,0)=0$,
 we have
 \[
   0= \F=  [f_u\cdot f_v]=
     vX(0,2)+O_2(u,v).
 \]
 In particular, $X(0,2)=0$.
 Using the fact 
 $X(j,k)=0$ ($j+k=2$), we have
 \begin{multline*}
  [f_u\cdot f_u]=
      1+u^2 \left(X(3,0)+4 Y(1,0)^2+Z(2,0)^2\right) \\
        +2 u v \left(X(2,1)+2 Y(0,1) Y(1,0) +Z(1,1)Z(2,0)\right)\\
        +v^2 \left(X(1,2)+Y(0,1)^2+Z(1,1)^2\right) +O_3(u,v).
 \end{multline*}
 By \eqref{eq:E0}, we have
 \begin{align}
  &X(3,0)+4 Y(1,0)^2+Z(2,0)^2=a_{2,0}^2, 
  \label{eq:1-1} 
  \\
  &X(2,1)+2 Y(0,1) Y(1,0) +Z(1,1)Z(2,0)=a_{2,0}a_{1,1},
  \label{eq:1-2} 
  \\
  &X(1,2)+Y(0,1)^2+Z(1,1)^2=1+a_{1,1}^2.
  \label{eq:1-3} 
 \end{align}
 By \eqref{eq:F0},
 we have
 \begin{align}
  &\frac{1}{2} X(2,1)+2 Y(0,1)Y(1,0)
   +Z(1,1) Z(2,0)=a_{2,0}a_{1,1}, 
  \label{eq:2-1}
  \\
  &X(1,2)+Y(0,1)^2+Z(1,1)^2+Z(0,2) Z(2,0)=
   1+a_{1,1}^2+a_{2,0}a_{0,2}, 
  \label{eq:2-2}
  \\
  &X(0,3)+2 Z(0,2) Z(1,1)=2a_{1,1}a_{0,2}.
  \label{eq:2-3}
 \end{align}
 Similarly, \eqref{eq:G0} yields
 \begin{align}
  \label{eq:3-1}
  &Y(0,1)^2+Z(1,1)^2=1+a_{1,1}^2,\\
  \label{eq:3-2}
  &Z(0,2) Z(1,1)=a_{1,1}a_{0,2},\\
  \label{eq:3-3}
  & Z(0,2)^2=a_{0,2}^2.
 \end{align}
 Since $Z(0,2)$ and $a_{0,2}$ are positive 
 (cf.\ Fact \ref{fact:hhnsuy} and \eqref{eq:bz-init}), 
 \eqref{eq:3-3} reduces to
 $Z(0,2)=a_{0,2}$.
 Then \eqref{eq:3-2} yields that
 $Z(1,1)=a_{1,1}$.
 Moreover, 
 \eqref{eq:3-1} reduces to
 $Y(0,1)=1$ because of $Y(0,1)>0$
 (cf.\ \eqref{eq:2nd-N}).
 On the other hand,
 \eqref{eq:2-3} implies
 $X(0,3)=0$.
 Also $X(1,2)=0$ follows from
 \eqref{eq:1-3}.
 Then \eqref{eq:2-2}
 yields $Z(2,0)=a_{2,0}$.
 Finally, \eqref{eq:1-2} and
 \eqref{eq:2-1} reduce to
 \[
    X(2,1)+2Y(1,0)=0,\quad
    X(2,1)+4Y(1,0)=0.
 \]
 So we have $X(2,1)=Y(1,0)=0$.
 Moreover, \eqref{eq:1-1} yields $X(3,0)=0$.
\end{proof}

Now we can prove 
Theorem \ref{thm:main} under the assumption that 
Proposition \ref{prop:key} is proved:
\begin{proof}[Proof of Theorem \ref{thm:main}]
 By Proposition \ref{prop:key}, we get formal
 power series $X$, $Y$, $Z\in \R[[u,v]]$ such that
 $F:=(X,XY+b(Y),Z)$ satisfies
 \begin{equation}\label{eq:EFG}
  \E=[F_u\cdot F_u], \quad
  \F=[F_u\cdot F_v],\quad
  \G=[F_v\cdot F_v]. 
 \end{equation}
 In fact, the coefficients $b_j$ ($j\ge m+1$) do not
 affect our computation, and
 \begin{align*}
  [F_u\cdot F_u]&=[F^m_u\cdot F^m_u]+O_{m+1}(u,v), \\
  [F_u\cdot F_v]&=[F^m_u\cdot F^m_v]+O_{m+1}(u,v), \\
  [F_v\cdot F_v]&=[F^m_v\cdot F^m_v]+O_{m+1}(u,v)
 \end{align*}
 hold, where $O_{m+1}(u,v)$ is a term in $\O_{m+1}$.
 Then by Borel's theorem, there exist $C^\infty$ functions
 $x$, $y$, $z$ whose Taylor series are $X$, $Y$, $Z$, respectively.
 So we set
 \[
 f(u,v):=\left(
              x(u,v),
              x(u,v)y(u,v)+b\bigl(y(u,v)\bigr),
              z(u,v)\right),
 \]
 then the first  fundamental form of $f$ is formally
 isometric to $d\sigma^2$.
 By Lemma \ref{lem:2},  the map $(u,v)\mapsto (x(u,v),y(u,v))$
 is a local diffeomorphism at the origin.
 Taking $(x,y)$ as a new local coordinate system,
 we can write $u=u(x,y)$ and $v=v(x,y)$.
 So $(x,y)$ gives the canonical coordinate system of
 the map $f$.
 Thus $f$ satisfies (1) and (2) of  Theorem \ref{thm:main}.
\end{proof}

\section{Properties of power series}\label{sec:Pw}
In this section,  we prepare  several properties of  power series 
to prove the case $m\ge 3$ of
Proposition \ref{prop:key}.
As in Section~\ref{sec:prelim}, we denote by
$\R[[u,v]]$ the ring of formal power series
with two variables $u,v$ in real coefficients.
Each element of $\R[[u,v]]$ can be written
as in \eqref{eq:P}. Each $P(k,l)$ ($k,l\ge 0$)
is called the \emph{$(k,l)$-coefficient}
of the power series $P$.
Moreover, the sum $k+l$ is called the
\emph{order} of the coefficient $P(k,l)$.
In particular, 
$P(k,l)$ $(k+l=m)$
consist of all coefficients of order $m$.
The formal partial derivatives of $P$ denoted by
\[
   P_u:=\partial P/\partial u,\qquad
   P_v:=\partial P/\partial v
\]
are defined in the usual manner.

\begin{lemma}\label{lem:diff}
 The $(k,l)$-coefficient of the {\rm(}formal{\rm)} partial
 derivatives $P_u$ and $P_v$ of $P$
 are given by
 \[
    P_u(k,l)=P(k+1,l),\qquad P_v(k,l)=P(k,l+1).
 \]
\end{lemma}
Linear operations on power series
also have a simple description as follows:
\begin{lemma}\label{lem:linear}
 Let $P$, $Q$ be two power series in  $\R[[u,v]]$,
 and let $\alpha$, $\beta\in \R$.
 Then 
 \[
   (\alpha P+\beta Q)(k,l)=\alpha P(k,l)+\beta Q(k,l).
 \]
\end{lemma}

The coefficient formula for products 
is as follows:
\begin{lemma}
\label{formula:power-prod}
 Let $P_1,\dots,P_N$ be power series in
 $\R[[u,v]]$. Then
 \begin{equation}\label{eq:m-prod}
    (P_1\cdots P_N)(k,l) = 
     k!l!
     \sum_{\substack{s_1+\dots+s_{N}=k,\\t_1+\dots+t_{N}=l}}
           \frac{P_1(s_1,t_1)\cdots P_s(s_{N},t_{N})}{
              s_1!t_1!\cdots s_{N}! t_{N}!}.
 \end{equation}
\end{lemma}
If $N=2$, and $P_1=P$ and $P_2=Q$, then
the formula \eqref{eq:m-prod}
reduces to the 
following:
\begin{equation}\label{eq:2-prod}
 (PQ)(k,l):=
  k!l!
  \sum_{s=0}^{k}\sum_{t=0}^{l}
  \frac{P(s,t)Q(k-s,l-t)}{s!t!(k-s)!(l-t)!}.
\end{equation}
Moreover, setting $Q$ to be the monomial
$u$ or $v$, we get the following:
\begin{corollary}\label{cor:power-prod}
 \[
    (u P)(k,l) = kP(k-1,l),\quad
    (v P)(k,l)=lP(k,l-1)
 \]
 hold, where coefficients with negative induces
 $P(-k,l)$, $P(m,-n)$ $(k,n>0, l,m\in \Z)$
 are considered as $0$.
\end{corollary}

In Definition \ref{def:order},
we defined the ideal
$\O_m$ of $\R[[u,v]]$
consisting of formal power series of order at least $m$.
The following assertion is obvious:
\begin{lemma}\label{lem:order-sum}
 If $P\in \O_n$ and
 $Q\in \O_m$, then
 $P+Q\in \O_r$,
 where $r=\min\{n,m\}$.
\end{lemma}
Let $P_j$ ($j=1,\dots,N$)
be a power series in
$\O_{n_j}$. 
For a given non-negative integer $m$,
we set
\[
   \innner{P_j}{P_1\cdots P_N}_m:=
               m-\sum_{k\ne j} n_k.
\]
Roughly speaking, this number is an 
upper bound of the degree of 
the terms of $P_j$
to compute the $m$-th order term of
$P_1\cdots P_N$, as follows.

\begin{proposition}\label{lem:order-prod}
 Let $P_j$ be a power 
 series in $\O_{n_j}$ $(j=1,\dots,N)$. Then 
 the product $P_1\cdots P_N$
 belongs to the class 
 $\O_{n_1+\cdots+n_N}$.
 Moreover, each $(k,l)$-coefficient
 $(P_1\cdots P_N)(k,l)$ $(k+l=m)$
 of order $m$ can be written 
 as a 
 homogeneous polynomial of 
 degree $N$ in the variables
 generated by
 \[
   \prod_{i=1}^N P_i(a_i,b_i) \qquad 
   \biggl(n_i\le a_i+b_i\le \innner{P_i}{P_1\cdots P_N}_m\biggr).
 \]
\end{proposition}

The proof of this assertion is not so difficult,
and so we leave it as an exercise.
As a consequence, we get the following:
\begin{corollary}\label{cor:pi}
 Let $\pi_m:\R[[u,v]]\to \R[[u,v]]/\O_{m+1}$ be
 the canonical homomorphism.
 The  $m$-th order term of $P_1\cdots P_N$
 are determined by 
 \[
    \pi_{m_1^{}}(P_1),\dots,\pi_{m_N^{}}(P_N),
 \]
 where
 \[
    m_i^{}:=\innner{P_i}{P_1\cdots P_N}_m \qquad
    (i=1,\dots,N).
 \]
 In other words, to compute
 the $m$-th order terms of 
 $P_1\dots P_N$, we need only the at most
 degree $m_i$ terms of $P_i$. 
\end{corollary}

\begin{example}\label{eq:PQ}
 We set $P\in \O_2$ and  $Q\in \O_1$ as
 \[
  P:=
    c_{1} u v+c_{2} v^2+c_{3} u^3 v+\cdots, \quad
  Q:=
   d_1 u+d_2 v+d_{3} u^3+\cdots,
 \]
 respectively.
 Then it holds that
 \[
    \innner{P}{PQ}_3=3-1=2,\qquad 
    \innner{Q}{PQ}_3=3-2=1.
 \]
 To compute $PQ$ modulo $\O_4$, we need
 the information of $\pi_2(P)$ and $\pi_1(Q)$.
 So, we have that 
 \begin{align*}
  PQ&=(c_{1} u v+c_{2} v^2)(d_1u+d_2v)+O_4(u,v)\\
  &=(c_{1} d_2+c_{2}d_1)uv^2+c_{1} d_1u^2v+c_{2}d_2v^3
  +O_4(u,v),
 \end{align*}
 and so $PQ\in \O_3$, where $O_j(u,v)$
 is an element of $\O_j$ ($j=1,2,3,\dots$).
 The coefficients of the terms of order $3$ 
 are
 \[
    c_{1} d_2+c_{2}d_1,\qquad c_{1} d_1,
       \qquad c_{2}d_2. 
 \]
 They are homogeneous polynomials of degree $2$
 in the variables $c_{1}$, $c_{2}$, $d_1$, $d_2$.
\end{example}

Let $X_1,\dots,X_r$ be power series in $\R[[u,v]]$.
We suppose that
for each $P_i$ ($i=1,\dots,N$)
there exists a unique number $\mu_{i}\in \{1,\dots,r\}$
such that
\begin{equation}\label{eq:Pi}
    P_i=X_{\mu_{i}},\,\, (X_{\mu_{i}})_u,
        \mbox{ or }\,\, (X_{\mu_{i}})_v.
\end{equation}
For the sake of simplicity, we set
\begin{equation}\label{eq:X}
   \check X_i:=X_{\mu_{i}}.
\end{equation}
In this case, each coefficient of the 
product $P_1\cdots P_N$
can be expressed as coefficients of
$X_1,\dots,X_r$. For each non-negative integer $m$,
we set
\[
   \innner{\check X_{i}}{P_i,P_1\cdots P_N}_{m}:=
   \begin{cases}
    \innner{P_i}{P_1\cdots P_N}_{m}  & 
    \mbox{if $P_i=\check X_{i}$}, \\
    \innner{P_i}{P_1\cdots P_N}_{m}+1
    & \mbox{if $P_i=(\check X_{i})_u \mbox{ or }
 (\check X_{i})_v$}.
   \end{cases}
\]
Roughly speaking, this number is an 
upper bound of the degree of 
the terms of $\check X_{i}$ appeared in $P_i$
to compute the $m$-th order term of
$P_1\cdots P_N$. 
In fact, by applying Lemma \ref{lem:order-prod}
and Lemma \ref{lem:diff} to this situation,
we get the following:

\begin{proposition}\label{cor:order-prod}
 Under the conventions
 \eqref{eq:Pi} and \eqref{eq:X},
 each $(k,l)$-coefficient
 $(P_1\dots P_N)(k,l)$ $(k+l=m)$
 of order $m$ is
 a homogeneous polynomial of 
 degree $N$ in the variables
 generated by
 \[
    \prod_{i=1}^N \check X_{i}(b,c) \qquad 
    \biggl(0\le b+c\le \innner{\check X_{i}}{P_i,P_1\cdots P_N}_m\biggr).
 \]
\end{proposition}

\begin{corollary}
 \label{cor:order-prod2}
 The  $m$-th order terms of $P_1\cdots P_N$
 are determined by 
 \[
 \pi_{m'_1}(\check X_{1}),\,\,\dots,\,\,
 \pi_{m'_n}(\check X_{N}),
 \]
 where
 \[
  m'_i:=
    \innner{\check X_{i}}{P_i,P_1\cdots P_N}_{m}
    \qquad
    (i=1,\dots,N).
 \]
 In other words, to compute
 $P_1\cdots P_N$, we need only the at most
 degree $m'_i$ terms of $\check X_{i}$. 
\end{corollary}
\begin{example}
 We set
 \[
   X := x_{1} u v + x_{2} v^2 + x_{3}  u^3 v+\cdots, \quad
   Y:= y_1 u + y_2 v + y_{3} u^3+\cdots,
 \]
 and 
 $P:=X_u$, $Q:=Y_u$.
 Then,
 \[
    P=x_{1}v+3x_{3} u^2 v +\cdots,\quad
    PQ=y_1+3 y_{3}u^2 +\cdots.
 \]
 Since $P\in \O_1$ 
 and $Q\in \O_0$,
 we have
 \begin{align*}
  \innner{X}{P,PQ}_1&=
  \innner{P}{PQ}_1+1
  =(1-0)+1=2, \\
  \innner{Y}{Q,PQ}_1&=
  \innner{Q}{PQ}_1+1
  =1.
 \end{align*}
 So we have
 \begin{align*}
  PQ&=\left(\pi^{}_2(X)\right)_u
  \left(\pi^{}_1(Y)\right)_{u}+O_2(u,v)\\
  &=
  \left(x_{1} u v + x_{2} v^2\right)_u
  \left(y_1 u + y_2 v\right)_u
  +O_2(u,v)
  =x_1y_1u+O_2(u,v).
 \end{align*}
 On the other hand,
 \begin{align*}
  \innner{X}{P,PQ}_2&=
  \innner{P}{PQ}_2+1
  =(2-0)+1=3, \\
  \innner{Y}{P,PQ}_2&=
  \innner{Q}{PQ}_2+1
  =(2-1)+1=2.
 \end{align*}
 So we have
 \begin{align*}
  PQ&=\left(\pi^{}_3(X)\right)_u\
  \left(\pi^{}_2(Y)\right)_{u}+O_3(u,v)\\
  &=
  \left(x_{1} u v + x_{2} v^2 + x_{3}  u^3 v\right)_u
  \left(y_1 u + y_2 v\right)_u
  +O_3(u,v)\\
  &=x_1y_1v+O_3(u,v).
 \end{align*}
 In this case, the upper bound 
 $\innner{Y}{Q,PQ}_2=2$
 of the order of $Y$
 for the contribution of the
 order $2$ coefficients of $PQ$
 is not sharp.
 In fact, there are no order $2$ terms for 
 $PQ$.
 Also, this does not contradict 
 Proposition \ref{cor:order-prod},
 since $0$ can be considered as a homogeneous polynomial 
 of order $2$
 whose coefficients are all zero.
\end{example}

\section{The ignorable terms when determining $f^m$}
\subsection{The leading terms and ignorable  terms}
Let $x(u,v)$, $y(u,v)$, $z(u,v)$ and $b(t)$ be $C^\infty$ functions
as in Lemma \ref{lem:realize-zero},
and $P$ a polynomial in 
\[
  x, \,\, x_u,\,\, x_v\quad y, \,\, y_u,
  \,\, y_v,\quad z, \,\, z_u,\,\, z_v,\quad b'(y).
\]
Each term (cf.\ \eqref{eq:term}) 
of $(f_u\cdot f_u)$, $(f_u\cdot f_v)$ or
$(f_v\cdot f_v)$ is a typical example of such 
polynomials.
We denote by $P|_m$ the finite formal power series
in $u,v$ (i.e.\ a polynomial in $u,v$) 
that results after the substitutions
\[
  x:=\Xx_{m+1},\quad y:=\Yy_{m-1}, \quad z:=\Zz_{m},
  \quad b(y):=\beta_{m+1}(\Yy_{m-1})
\]
into $P$. 
\begin{definition}
 A term $T$ of $(f_u\cdot f_u)$, $(f_u\cdot f_v)$
 or $(f_v\cdot f_v)$ 
 is called an \emph{$m$-ignorable term}
 ($m\ge 3$)
 if each $m$-th order coefficient
 \[
    (T|_m)(j,k)\qquad (j+k=m)
 \]
 does not contain any of the top term
 coefficients 
 \begin{align*}
  &X(j,k) \qquad (j+k=m+1), \\
  &Y(j,k) \qquad (j+k=m-1),\\ 
  &Z(j,k) \qquad (j+k=m)
 \end{align*}
 of $\Xx_{m+1}$, $\Yy_{m-1}$, $\Zz_{m}$
 (cf.\ \eqref{eq:X00}, \eqref{eq:Y00} and \eqref{eq:Z00}).
 A term which is not $m$-ignorable
 is called a \emph{leading term} of order $m$.
\end{definition}
For the computation of leading terms,
we will use the following two convenient
equivalence relations:
Let $\A$ be the associative algebra
generated by 
\[
   X(j,k),\quad Y(j,k),\quad Z(j,k) \qquad (j,k=0,1,2,\dots).
\]
We denote by $\A_m$ the ideal of $\A$
generated by
\begin{align*}
 &X(j,k) \qquad (j+k\le m+1), \\
 &Y(j,k) \qquad (j+k\le m-1),\\ 
 &Z(j,k) \qquad (j+k\le m).
\end{align*}
If two elements $\delta_1$, $\delta_2\in \A$
satisfy $\delta_1-\delta_2\in \A_{m-1}$, then 
we write
\begin{equation}\label{eq:equiv0}
 \delta_1\equiv_m \delta_2.
\end{equation}

Let $P$, $Q$ be two polynomials in 
\[
  x, \,\, x_u,\,\, x_v\quad y, \,\, y_u,\,\, y_v,\quad z, \,\, z_u,\,\,
  z_v,
  \,\, b'(y).
\]
If all of the coefficients of $P|_m-Q|_m$
(as a polynomial in $u,v$)
are contained in $\A_{m-1}$, we denote this by
\[
   P\equiv_m Q.
\]
This notation is the same as the one used in 
\eqref{eq:equiv0}, and this is
rather useful for unifying the symbols.
For example, if the term $T$ satisfies
\[
   T \equiv_m 0
\]
if and only if the term $T$ is $m$-ignorable.

\subsection{The properties of terms containing $b'(y)$}
In the right hand sides of
\eqref{eq:E}, \eqref{eq:F} and \eqref{eq:G},
terms containing $b'$ appear,
and they are
\begin{gather}
 \label{eq:b1}
   b'(y)y x_u y_u,\quad
   b'(y)x y_u^2,\quad
   b'(y)^2 y_u^2\qquad (\mbox{in $f_u\cdot f_u$}),\\
 \label{eq:b2}
   b'(y)y x_u y_v,\quad
   b'(y)y x_v y_u,\quad
   b'(y)x y_uy_v,\quad
   b'(y)^2 y_uy_v \qquad (\mbox{in $f_u\cdot f_v$}),\\
 \label{eq:b3}
   b'(y)y x_v y_v,\quad
   b'(y)x y_v^2,\quad
   b'(y)^2 y_v^2 \qquad (\mbox{in $f_v\cdot f_v$}),
 \end{gather}
respectively.

In this subsection, we show that
the terms as in 
\eqref{eq:b1}, \eqref{eq:b2} and \eqref{eq:b3} are all
$m$-ignorable.
To prove this,
we prepare the following lemma:
\begin{lemma}\label{lem:beta}
 Each $m$-th order coefficient
 of the power series 
 associated to $b'(y(u,v))$
 can be expressed 
 as a polynomial in 
 the variables $Y(k,l)$ $(k+l\le m-1)$.
\end{lemma}
 \begin{proof}
  We can write (cf.\ \eqref{eq:beta})
  \[
    B:=[b']=\sum_{r=2}^\infty \frac{B(r)}{r!} t^r.
  \]
  Since $Y\in \O_1$ and the index satisfies $r\ge 2$, we have
  \[
     \innner{Y}{Y^r}_m=m-(r-1)\le m-1.
  \]
  By Lemma \ref{lem:order-prod}, each
  $m$-th order term of $Y^r$
  is a homogeneous polynomial of degree $r$
  in variables $Y(k,l)$ ($k+l\le m-1$)
  for all $r\ge 2$.
  In particular, each $m$-th order term of
  the power series induced by
  \[
    \B(Y)=\sum_{r=2}^\infty \frac{B(r)}{r!}Y^r
  \]
  can be expressed as a  polynomial
  of degree at most $m$
  in variables $Y(k,l)$ ($k+l\le m-1$),
  which proves the assertion.
\end{proof}

In this situation, we set
\[
   \innner{Y}{\B(Y),P_1 \cdots P_N}_m :=
   \innner{\B(Y)}{P_1 \cdots P_N}_m - 1.
\]
Then, this number gives an 
upper bound of the degree of 
the terms of $Y$ appeared in
$\B(Y)$
to compute the $m$-th order term of
$P_1\cdots P_N$. 
\begin{proposition}
 The terms given in \eqref{eq:b1}, \eqref{eq:b2}
 and \eqref{eq:b3} are all
 $m$-ignorable.
\end{proposition}
\begin{proof}
 We can categorize
 the terms in \eqref{eq:b1},
 \eqref{eq:b2} and \eqref{eq:b3}
 into three classes.
 One is
 \begin{align}\label{cat:B1}
  b'(y)y x_uy_u\in \O_2\O_1\O_0\O_1,\quad
  b'(y)y x_uy_v\in \O_2\O_1\O_0\O_0,\\
  \nonumber
  b'(y)y x_vy_u\in \O_2\O_1\O_3\O_1,\quad 
  b'(y)y x_vy_v\in \O_2\O_1\O_3\O_0, 
 \end{align}
 and the other two classes are
 \begin{equation}
  \label{cat:B2}
   b'(y)x y_u^2\in \O_2\O_1\O_1\O_1,~
   b'(y)x y_uy_v\in \O_2\O_1\O_1\O_0,~
   b'(y)x y_v^2\in \O_2\O_1\O_0\O_0,
 \end{equation}
 and
 \begin{equation}\label{cat:B3}
  b'(y)^2y_u^2\in \O_2\O_2\O_1\O_1,~
   b'(y)^2y_uy_v\in \O_2\O_2\O_1\O_0,~
   b'(y)^2 y_v^2\in \O_2\O_2\O_0\O_0,
 \end{equation}
 respectively.

 We consider the term
 $b'(y)y x_uy_v$ in \eqref{cat:B1}.
 We have 
 \[
 \innner{Y}{\B(Y),\B(Y)Y X_u Y_v}_m
 =(m-1-0-0)-1=m-2.
 \]
 Since $m-2$ is less than $m-1$, $\B(Y)$
 does not effect the leading term
 (cf.\ Lemma~\ref{lem:beta}).
 Similarly, we have that
 \begin{align*}
  \innner{Y}{\B(Y)Y X_u Y_v}_m
  &=m-2-0-0=m-2\,\,(<m-1), \\
  \innner{X}{X_u,\B(Y)Y X_u Y_v}_m
  &=(m-2-1-0)+1
  =m-2\,\,(<m+1), \\
  \innner{Y}{Y_v,\B(Y)Y X_u Y_v}_m
  &=(m-2-1-0)+1
  =m-2\,\,(<m-1) ,
 \end{align*}
 and we can conclude that 
 $b'(y)y x_u y_v$ is an $m$-ignorable term.
 Similarly, one can also prove that other 
 three terms in \eqref{cat:B1}
 are also $m$-ignorable.
 
 \medskip
 We next consider the term
 $b'(y)x y_{v}^2$ in \eqref{cat:B2}.
 We have
 \[
 \innner{Y}{\B(Y),\B(Y)X Y_v^2}_m
 =(m-1-0-0)-1=m-2\,\,(< m-1).
 \]
 Hence, the coefficients of $Y$ 
 appeared in $\B(Y)$
 do not effect the leading term.
 Similarly, the facts
 \begin{align*}
  \innner{X}{\B(Y)X Y_v^2}_m
  &=m-2-0-0=m-2 \,\,(<m), \\
  \innner{Y}{Y_v,\B(Y)X Y_v^2}_m
  &=(m-2-1-0)+1
  =m-2\,\,(<m-1), 
 \end{align*}
 imply that the term
 $b'(y)x y_v^2$ is an $m$-ignorable term.
 Similarly, other two terms in \eqref{cat:B2}
 are also $m$-ignorable.

\medskip
 Finally, we consider 
 the term
 $b'(y)^2y_v^2$ in  \eqref{cat:B3}.
 Since
 \[
    \innner{Y}{\B(Y),\B(Y)^2Y_v^2}_m
     =(m-2-0-0)-1=m-3\ (< m-1),
 \]
 the coefficients of $Y$ 
 appearing in $b'(Y)$
 do not effect the leading term.
 We have
 \[
     \innner{Y}{Y_v,\B(Y)^2Y_v^2}_m
        =(m-2-2-0)+1=m-3\ (<m-1),
 \]
 and can conclude that
 $b'(y)^2y_v^2$ is 
 an $m$-ignorable term.
 Similarly,
 other two terms 
 in  \eqref{cat:B3} are also
 $m$-ignorable.
\end{proof}

The following terms
\begin{gather}
\label{eq:c1}
 x^2 y_u^2 \qquad (\mbox{in $f_u\cdot f_u$}),
\\
\label{eq:c2}
 x_ux_vy^2
 ,\quad
 xyx_vy_u\qquad (\mbox{in $f_u\cdot f_v$}),\\
\label{eq:c3}
 xyx_vy_v,\quad
 x_v^2,
 \quad
 x_v^2y^2\qquad (\mbox{in $f_v\cdot f_v$})
\end{gather}
appear in
\eqref{eq:E},
\eqref{eq:F} and
\eqref{eq:G}.
We show the following:
\begin{proposition}\label{prop:rest}
 The terms given in 
 \eqref{eq:c1}, \eqref{eq:c2}
 and \eqref{eq:c3} are all
 $m$-ignorable terms.
\end{proposition}
\begin{proof}
 Since $x^2 y_u^2\in \O_1\O_1\O_1\O_1$,
 we have
 \begin{align*}
  \innner{X}{X^2Y_u^2}_m&=m-1-1-1=m-3\ (<m+1),\\
  \innner{Y}{Y_u,X^2Y_u^2}_m&=(m-1-1-1)+1=m-2\ (<m-1).
 \end{align*}
 This  implies that $x^2 y_u^2$ is 
 $m$-ignorable.
 On the other hand,
 $x_ux_vy^2$
 and
 $x_v^2y^2$
 both consist of two derivatives of $x$
 and $y^2$,
 and the former term has lower total order,
 so if we show $x_ux_vy^2$ is 
 $m$-ignorable,
 then so is $x_v^2y^2$. 
 In fact, since $x_ux_vy^2\in \O_0\O_3\O_1\O_1$, 
 we have
 \begin{align*}
  \innner{X}{X_u,X_uX_vY^2}_m&\le 
  \innner{X}{X_v,X_uX_vY^2}_m
  \\
  &=m-0-1-1+1=m-1\ (<m+1),\\
  \innner{Y}{X_uX_vY^2}_m&=m-0-3-1=m-4\ (<m-1).
 \end{align*}
 So $x_ux_vy^2$ is
 $m$-ignorable. 

 Finally,
 $xyx_vy_u$ and
 $xyx_vy_v$
 both consist of $xy$ and derivatives of $x,y$.
 The term $xyx_vy_v$ has lower total order.
 So if it is $m$-ignorable, then so is $xyx_vy_u$.
 The fact that $xyx_vy_v \in \O_1\O_1\O_3\O_0$
 is $m$-ignorable follows from the
 following computations:
 \begin{align*}
  \innner{X}{XYX_vY_v}_m&=
  \innner{Y}{XYX_vY_v}_m= 
  m-1-3-0=m-4\,\,(<m-1),\\
  \innner{X}{X_v,XYX_vY_v}_m&=
  (m-1-1-0)+1=m-1\,\,(<m+1),\\
  \innner{Y}{Y_v,XYX_vY_v}_m&=
  (m-1-1-3)+1=m-4\,\,(<m-1).
 \end{align*}
 Finally, $x_v^2\in \O_3\O_3$
 is $m$-ignorable because 
 $\innner{X}{X_v,X_v^2}_m=m-2\,\,(<m+1)$.
\end{proof}

\section{The existence of the formal power series solution}
\label{sec:App}
\subsection{%
  Leading terms of $(f_u\cdot f_u)$, 
   $(f_u\cdot f_v)$ and $(f_v\cdot f_v)$.}
Applying the computations 
in the previous section, we prove
the following:
\begin{proposition}
 Let $m$ be an integer greater than $2$,
 and $k$, $l$ non-negative integers
 such that
 \begin{equation}\label{eq:m}
  k+l=m\geq 3.
 \end{equation}
 Then
 the $m$-th order terms of the equalities
 \eqref{eq:E0}, \eqref{eq:F0}
 and \eqref{eq:G0}
 reduce to the following relations{\rm:}
 \begin{align}
  \label{eq:Eee}
  \E(k,l) &\equiv_m
  2\biggl(
                X(k+1,l)+
                (k+1)l Y(k,l-1)\\
  \nonumber
  &\hspace{0.3\linewidth}
           +k a_{2,0} Z(k,l) + la_{1,1}Z(k+1,l-1)
            \biggr),
  \\
  \label{eq:Fff}
  \F(k,l) &\equiv_m  X(k,l+1)+
            mk Y(k-1,l)+
            k a_{2,0} Z(k-1,l+1)\\
  \nonumber
  &\hspace{0.3\linewidth}
  +m a_{1,1}Z(k,l)+ la_{0,2}Z(k+1,l-1), \\
  \label{eq:Ggg}
  \G(k,l) &\equiv_m  2\biggl(
                         k(k-1)Y(k-2,l+1)
  \\ \nonumber 
  &\hspace{0.3\linewidth}
    +    ka_{1,1}Z(k-1,l+1)+la_{0,2}Z(k,l)\biggr),
 \end{align}
 where
 \begin{multline*}
  Y(m,-1)=Y(-1,m)= Y(-2,m+1)=Y(m+1,-2)\\
   = Z(-1,m+1)=Z(m+1,-1)=0.
 \end{multline*}
\end{proposition}
\begin{proof}
 Removing the $m$-ignorable terms
 \eqref{eq:b1} and \eqref{eq:c1}
 from $[f_u\cdot f_u]$ in \eqref{eq:E},
 we have
 \[
   [f_u\cdot f_u]
     \equiv_m L_1,\qquad
          L_1:=X_u^2+Y^2 X_u^2+2 XYX_uY_u+Z_u^2.
 \]
 Similarly, we get 
 from \eqref{eq:F}, \eqref{eq:G}, \eqref{eq:b1} and \eqref{eq:c1}
 that 
 \[
    [f_u\cdot f_v]
         \equiv_m L_2 ,\qquad
    [f_v\cdot f_v]
        \equiv_m L_3, 
 \]
 where 
 \[
        L_2:=X_uX_v+XYX_uY_v
              + X^2Y_uY_v+Z_uZ_v, \qquad
        L_3:=X^2Y_v^2+Z_v^2.
 \]
 The first term of $L_1$ is $X_u^2$.
 We can write 
 $X=u+\tilde X$ ($\tilde X\in \O_4$).
 Since 
 \[
    \innner{\tilde X}{\tilde X_u,\tilde X_u^2}_m=m-2\ (<m+1),
 \]
 we have
 \[
    X_u^2= (1+\tilde X_u)^2\equiv_m 2 \tilde X_u
 \]
 and 
 \begin{equation}\label{eq:I1}
  X_u^2(k,l)\equiv_m 2 \tilde X_u(k,l)
   \equiv_m 2 X_u(k,l)=2X(k+1,l)\qquad
   (k+l=m\geq 3),
 \end{equation}
 where we have applied Lemma \ref{lem:diff}.
 The second term of $L_1$ is
 $Y^2X_u^2$.
 Since
 \[
    \innner{X}{X_u,Y^2X_u^2}_m=m-1-1-0+1=m-1\,\,(<m+1)
 \]
 and
 \[
    \innner{Y}{Y^2X_u^2}_m=m-1-0-0=m-1,
 \]
 the $m$-th order terms of $Y^2X_u^2$
 might not be $m$-ignorable.
 In fact, it 
 can be written in terms 
 of the $(m-1)$-st order coefficients of $Y$
 as follows.
 Since 
 $X_u=1+\tilde X_u$ ($\tilde X_u\in \O_3$),
 we have
 $Y^2X_u^2\equiv_m Y^2$.
 Since $Y=v+\tilde Y$ ($\tilde Y\in \O_2$)
 and $\innner{\tilde Y}{\tilde Y^2}_m=m-2\ (<m-1)$,
 $\tilde Y^2$ is an 
 $m$-ignorable term.
 Thus, we have
 \[
    Y^2\equiv_m
     v^2+2v\tilde Y+\tilde Y^2  
     \equiv_m
     2v \tilde Y,
 \]
 and so
 \begin{equation}\label{eq:I2}
  (Y^2X_u^2)(k,l)
   \equiv_m 2(v \tilde Y)(k,l)\equiv_m 2(v Y)(k,l)\equiv_m 2l Y(k,l-1),
 \end{equation}
 where we have applied Corollary \ref{cor:power-prod}.
 We examine the third term $XYX_uY_u$
 of $L_1$.
 Since
\begin{align*}
 \innner{X}{XYX_uY_u}_m&=m-1-0-1=m-2\ (<m+1),\\
 \innner{Y}{XYX_uY_u}_m&=m-1-0-1=m-2\ (<m-1),\\
 \innner{X}{X_u,XYX_uY_u}_m
 &=(m-1-1-1)+1=m-2\ (<m+1),\\
 \innner{Y}{Y_u,XYX_uY_u}_m
 &=(m-1-1-0)+1=m-1,
\end{align*}
 the $m$-th order terms of $XYX_uY_u$
 can be written in terms 
 of the coefficients of $Y_u$
 modulo $\A_{m-1}$.
 Thus
 \[
   XYX_uY_u=(u+\tilde X)(v+\tilde Y)(1+\tilde X_u)Y_u
   \equiv_m uvY_u
 \]
 and
 \begin{equation}\label{eq:star2}
  (XYX_uY_u)(k,l)\equiv_m(uvY_u)(k,l)=
   klY_u(k-1,l-1)\equiv_mklY(k,l-1).
 \end{equation}
 The fourth term of $L_1$ is $Z_u^2$.
 Since $Z_u^2\in \O_1\O_1$, we have
 \[
   \innner{Z}{Z_u,Z_u^2}_m
    =(m-1)+1=m.
 \]
 Hence the $m$-th order terms of $Z_u^2$
 can be written in terms of the coefficients of $Z_u$
 modulo $\A_{m-1}$.
 If we write 
 \[
   Z=\frac12(a_{2,0}u^2+2a_{1,1}uv+a_{0,2}v^2)+\tilde Z
 \qquad
 (\tilde Z\in \O_3),
 \]
 then $\tilde Z_u^2$ is an 
 $m$-ignorable term,
 and
 \[
   Z_u^2\equiv_m (a_{2,0}u+a_{1,1}v+\tilde Z_u)^2
    \equiv_m 2\left(a_{2,0}u\tilde Z_u+a_{1,1}v\tilde Z_u\right).
 \]
 Since $Z(k,l)=\tilde Z(k,l)$ for $k+l\ge 3$,
 we have
 \begin{align}\label{eq:I3}
  Z_u^2(k,l)
  &\equiv_m
  2a_{2,0}(uZ_u)(k,l)+2a_{1,1}(vZ_u)(k,l) \\
  &\nonumber
  =
  2a_{2,0}kZ_u(k-1,l)+2a_{1,1}lZ_u(k,l-1)\\
  &\nonumber
  =
  2ka_{2,0}Z(k,l)+2la_{1,1}Z(k+1,l-1).
 \end{align}
 By \eqref{eq:I1}, \eqref{eq:I2},
 \eqref{eq:star2}  and \eqref{eq:I3},
 we have \eqref{eq:Eee}.
 
 We next prove \eqref{eq:Fff}.
 Since
 \begin{align*}
  &\innner{X}{X_u,X_uX_v}_m=m-3+1=m-2
  \,\, (<m+1),\\
  &\innner{X}{X_v,X_uX_v}_m=m-0+1=m+1,
 \end{align*}
 $X_{v}$ contributes to the leading term.
 Thus
 \[
    X_uX_v\equiv_m (1+\tilde X_u)X_v \equiv_m X_v
 \]
 and
 \begin{equation}\label{II1}
  (X_uX_v)(k,l)\equiv_m X_v(k,l)=X(k,l+1).
 \end{equation}
 On the other hand, 
 since
 \begin{align*}
  \innner{X}{XYX_uY_v}_m&=m-1-0-0=m-1\ (<m+1),\\
  \innner{Y}{XYX_uY_v}_m&=m-1-0-0=m-1,\\
  \innner{X}{X_u,XYX_uY_v}_m
  &=(m-1-1-0)+1=m-1\ (<m+1),\\
  \innner{Y}{Y_v,XYX_uY_u}_m
  &=(m-1-1-0)+1=m-1,
 \end{align*}
 the coefficients of
 the factors $Y$ and $Y_v$
 appear in the leading terms. 
 Thus
 \[
    XYX_uY_v\equiv_m(u+\tilde X)(v+\tilde Y)(1+\tilde X_u)
    (1+\tilde Y_v)
    \equiv_m u (v+\tilde Y)(1+\tilde Y_v)
     \equiv_m uv \tilde Y_v+u \tilde Y
 \]
 and
 \begin{multline}\label{II2}
  (XYX_uY_v)(k,l)\equiv_m
  (uv Y_v)(k,l)+(u Y)(k,l) \\
  \equiv_m
  klY_v(k-1,l-1)+kY(k-1,l)
  \equiv_m
  k(l+1)Y(k-1,l).
 \end{multline}
 The third term of $L_2$ is
 $X^2Y_uY_v$.
 Since
 \begin{align*}
  \innner{X}{X^2Y_uY_v}_m&=m-1-1-0=m-2\ (<m+1),\\
  \innner{Y}{Y_u,X^2Y_uY_v}_m&=m-1-1-0+1=m-1,\\
  \innner{Y}{Y_v,X^2Y_uY_v}_m&=m-1-1-1+1=m-2\ (<m-1),
 \end{align*}
 only the factor $Y_u$
 affects the computation of the leading term.
 So we have
 \[
    X^2Y_uY_v\equiv_m (u+\tilde X)^2Y_u(1+\tilde Y_v)
      \equiv_m u^2 Y_u
 \]
 and
 \begin{multline}\label{II3}
  (X^2Y_uY_v)(k,l)\equiv_m (u^2 Y_u)(k,l)\\
  \equiv_m k(k-1)Y_u(k-2,l)
   \equiv_m k(k-1)Y(k-1,l).
 \end{multline}
 The fourth term of $L_2$ is $Z_uZ_v\in \O_1\O_1$.
 Since 
 \[
    \innner{Z}{Z_u,Z_uZ_v}_m=\innner{Z}{Z_v,Z_uZ_v}_m
     =m-1+1=m,
 \]
 $Z_u$ and $Z_v$ both contribute to the leading terms,
 and
 \begin{align*}
  Z_uZ_v
  &= 
  (a_{2,0}u+a_{1,1}v+\tilde Z_u)
  (a_{1,1}u+a_{0,2}v+\tilde Z_v)\\
  &\equiv_m 
  (a_{1,1}u+a_{0,2}v)\tilde Z_u+
  (a_{2,0}u+a_{1,1}v)\tilde Z_v. 
 \end{align*}
 So we have
 \begin{align}\label{II4}
  &(Z_uZ_v)(k,l)\equiv_m 
  \bigl((a_{1,1}u+a_{0,2}v)Z_u
  \bigr)(k,l)+
  \bigl((a_{2,0}u+a_{1,1}v)Z_v\bigr)(k,l) \\
  \nonumber
  &\equiv_m ka_{1,1}Z(k,l)+l a_{0,2}Z(k+1,l-1) 
  +ka_{2,0}Z(k-1,l+1)+la_{1,1}Z(k,l) \\
  \nonumber
  &=ka_{2,0}Z(k-1,l+1)
  +
  m a_{1,1}Z(k,l)+l a_{0,2}Z(k+1,l-1). 
 \end{align}
 By \eqref{II1}, \eqref{II2}, \eqref{II3} and \eqref{II4},
 we obtain \eqref{eq:Fff}.
 
 Finally, we consider $L_3$. 
 The first term of $L_3$ is
 $X^2Y_v^2$. 
 Since
 \begin{align*}
  \innner{X}{X^2Y_v^2}_m&=m-1-0-0=m-1\ (<m+1),\\
  \innner{Y}{Y_v,X^2Y_v^2}_m&=m-1-1-0+1=m-1,
 \end{align*}
 the coefficients of $Y_v$ affect the leading term
 of $X^2Y_v^2$.
 We have
 \[
    (X^2Y_v^2)\equiv_m
    (u+\tilde X)^2(1+\tilde Y_v)^2\equiv_m
    2u^2\tilde Y_v
 \]
 and 
 \begin{equation}\label{III1}
  (X^2Y_v^2)(k,l)\equiv_m
   2\bigl(u^2Y_v\bigr)(k,l)\equiv_m 2k(k-1)Y(k-2,l+1).
 \end{equation}
 The second term of
 $L_3$ is $Z_v^2$.  
 Since
 \[
  Z_v^2\equiv_m (a_{1,1}u+a_{0,2}v+\tilde Z_v)^2
   \equiv_m2(a_{1,1}u+a_{0,2}v)\tilde Z_v,
 \]
 we have
 \begin{equation}\label{III2}
  Z_v^2(k,l)\equiv_m 
   2\bigl((a_{1,1}u+a_{0,2}v)Z_v\bigr)(k,l)
   \equiv_m 2ka_{1,1}Z(k-1,l+1)+2la_{0,2}Z(k,l).
 \end{equation}
 By \eqref{III1} and \eqref{III2},
 we obtain \eqref{eq:Ggg}.
\end{proof}

\subsection{Proof of Proposition~\ref{prop:key}}
We prove the assertion by induction.
By Lemma \ref{lem:2},
we have already
determined the coefficients
\[
   X(i,l)\quad (0\le i+l\le 3), \quad
   Y(j,l)\quad (0\le  j+l\le 1), \quad
   Z(k,l)\quad (0\le  k+l\le 2). 
\]
For the sake of simplicity, we set
\begin{equation}\label{eq:index-simple}
    X_{i}:=X(i,m-i+1),\quad
    Y_{j}:=Y(k,m-j-1),\quad
    Z_{k}:=Z(k,m-k).
\end{equation}
We say that $W=X_i$, $Y_j$, $Z_k$ is \emph{$m$-fixed}
if it is expressed in terms of
\begin{align*}
 X(i,l)\qquad &(0\le i+l\le  m), \\
 Y(j,l)\qquad &(0\le  j+l\le  m-2), \\
 Z(k,l)\qquad &(0\le  k+l\le m-1), \\
 \E(i,l),\,\,\F(i,l),\,\,\G(i,l)\qquad& (0\le i+l\le  m).
\end{align*}
To prove the theorem,
it is sufficient to show that
\[
  X_i\quad (i=0,\dots,m+1),\qquad Y_j\quad (j=0,\dots,m-1), \qquad
  Z_k\quad (k=0,\dots,m)
\]
are all $m$-fixed.
By \eqref{eq:Eee}, 
\eqref{eq:Fff} and \eqref{eq:Ggg}, 
we can write
\begin{align}\label{eq:Eee1}
 X_{k+1}+(k+1)(m-k) Y_k
  +k a_{2,0} Z_k + (m-k)a_{1,1}Z_{k+1}
 =\tilde{\E}_k,
 \\
 \label{eq:Fff1}
 X_k+
            mk Y_{k-1}+
            k a_{2,0} Z_{k-1}
            +m a_{1,1}Z_k+ (m-k)a_{0,2}Z_{k+1}
 =\tilde {\F}_k,
 \\
 \label{eq:Ggg1}
                k(k-1)Y_{k-2}
 +
                ka_{1,1}Z_{k-1}+(m-k)a_{0,2}Z_k=\tilde {\G}_k,
\end{align}
for $k=0,\dots,m$, where
$\tilde {\E}_k$, $\tilde {\F}_k$ and $\tilde {\G}_k$
are all previously $m$-fixed terms, 
by the inductive assumption.

If we set $k=0$ in \eqref{eq:Ggg1},
we have 
\begin{equation}\label{eq:Z0}
 Z_0=\frac{\tilde {\G}_0}{m a_{0,2}},
\end{equation}
where we used the fact that $a_{0,2}>0$.
If we next set $k=1$ 
in \eqref{eq:Ggg1},
then we have
\[
   a_{1,1}Z_0+(m-1)a_{0,2}Z_1=\tilde {\G}_1
\]
and
\begin{equation}\label{eq:Z1}
 Z_1=\frac{\tilde{\G}_1-a_{1,1}Z_0}{(m-1)a_{0,2}}.
\end{equation}
Hence $Z_1$ is $m$-fixed (cf.\ \eqref{eq:Z0}).
On the other hand, \eqref{eq:Ggg1} for $2\le k\le m$
can be rewritten as
\begin{equation}\label{eq:PP}
 (1 + k) (2 + k) Y_k + a_{1,1} (2 + k) Z_{k+1} + 
  a_{0,2} (-2 - k + m) Z_{k+2} = \tilde {\G}_{k+2}
\end{equation}
for $k=0,\dots,m-2$.
If we set $k=0$ in \eqref{eq:Fff1},
then we have
\begin{equation}\label{eq:X0}
 X_0+m a_{1,1}Z_0+ma_{0,2}Z_1=\tilde {\F}_0.
\end{equation}
Thus $X_0$ can be $m$-fixed.
On the other hand, \eqref{eq:Fff1} for $1\le k\le m$
can be rewritten as
\begin{equation}\label{eq:Fff2}
 X_{k+1}+
            m(k+1) Y_{k}+
            (k+1) a_{2,0} Z_{k}
	    +m a_{1,1}Z_{k+1}+ (m-k-1)a_{0,2}Z_{k+2}
	    =\tilde {\F}_{k+1},
\end{equation}
where $k=0,\dots,m-1$.
Subtracting \eqref{eq:Eee1} from \eqref{eq:Fff2},
we have
\begin{equation}\label{eq:Fff3}
 k(k+1) Y_{k}+
           a_{2,0} Z_{k}
	   +ka_{1,1}Z_{k+1}+ (m-k-1)a_{0,2}Z_{k+2}
	   =\tilde {\F}_{k+1}
	   -\tilde {\E}_k
\end{equation}
for $k=0,\dots,m-1$.
By \eqref{eq:Fff3} and 
\eqref{eq:PP}, we have
\begin{multline}\label{imp}
 Z_{k+2}\\=\frac{1}{a_{0,2} (2 m-k-2)}
 \biggl(
 -a_{2,0} (2 + k) Z_k +(k+2)(
 \tilde {\F}_{k+1}
 -\tilde {\E}_k)-k 
 \tilde {\G}_{k+2}
 \biggr)
\end{multline}
for $k=0,\dots,m-2$.
Thus $Z_2,\dots,Z_m$ are $m$-fixed.
Then $Y_0,\dots,Y_{m-1}$
are $m$-fixed by
\eqref{eq:Fff3},
and $X_1,\dots,X_{m}$
are also $m$-fixed by
\eqref{eq:Fff2}.
Finally, if we set $k=m$ in \eqref{eq:Eee1},
then we have
\begin{equation}\label{eq:Xm+1}
 X_{m+1}
  +m a_{2,0} Z_m =\tilde {\E}_m,
\end{equation}
and $X_{m+1}$ is $m$-fixed.
Since the terms 
$\tilde {\E}_k$, $\tilde {\F}_k$
and $\tilde {\G}_{k}$
as in \eqref{eq:Eee1},
\eqref{eq:Fff1} and \eqref{eq:Ggg1}
can be computed using the lower order
term of $\Xx_{m+1}$, $\Yy_{m-1}$, $\Zz_{m}$
and the informations of $\E$, 
$\F$ and $\G$,
we  get a desired $m$-th formal solution
\[
  f^m=\bigl(\Xx_{m+1},
            \Xx_{m+1}\Yy_{m-1}
             +\beta_{m+1}(\Yy_{m-1}),
	     \Zz_{m})
\]
satisfying \eqref{eq:E0}, \eqref{eq:F0} and \eqref{eq:G0}.
The uniqueness of $f^m$ is now obvious from
our construction.

\begin{remark}
 If needed, we can explicitly write down
 the lower order terms 
 $\tilde {\E}_k$, $\tilde {\F}_k$
 and $\tilde {\G}_{k}$ as in \eqref{eq:Eee1},
 \eqref{eq:Fff1} and \eqref{eq:Ggg1},
 but we omit such formulas here, as they are
 complicated.
\end{remark}

\section{New intrinsic invariants of cross caps}\label{sec:5}

Let $\W$ be the set of
germs of Whitney metrics at their singularities.
Two metric germs $d\sigma^2_i$ ($i=1,2$) in $\W$
are called \emph{isometric} if there exists a local diffeomorphism
germ $\phi$ such that $d\sigma^2_2$
is the pull-back of $d\sigma^2_1$ by $\phi$.
A map
\[
   I:\W\to \R
\]
is called an \emph{invariant of Whitney metrics} if it
takes a common value for all metrics in each isometric class.
For a cross cap singularity, we can take a canonical coordinate system
$(x,y)$ such that
$f(x,y)$
is expressed as (cf.\ \eqref{eq:cross})
\begin{align} \label{eq:fn}
 f(x,y)&=
        \left(
	   x,\,\,
	   xy+b(y),\,\,
	   z(x,y)\right), \\
 b(y)&=\sum_{i=3}^\infty \frac{b_i}{i!} y^i, \quad
 z(x,y)=\sum_{j+k\ge 2}^\infty \frac{a_{j,k}}{j!k!} x^jy^k.
\nonumber
\end{align}
As shown in \cite[Theorem 6]{HHNUY}, 
the coefficients $a_{2,0}$, $a_{1,1}$ and
$a_{0,2}$ are intrinsic invariants.
By 
 \eqref{eq:comp-E},
 \eqref{eq:comp-F}
and
 \eqref{eq:comp-G}, one can
observe that these three invariants are determined 
by the second order jets of $E,F$ and $G$.
So one might expect that the coefficients of the Taylor expansions
of the functions $E,F,G$
are all intrinsic invariants of cross caps.
However, for example,
\[
  E_{vvv}(0,0)=6 a_{1,1}a_{1,2}
\]
is not an intrinsic invariant of $f$,
since
$a_{1,2}$ is changed by an isometric deformation of cross caps
(cf.\ \cite[Theorem 4]{HHNUY}).
In this section, we construct a family of intrinsic invariants
$\{\alpha_{i,j}\}_{i+j\ge 2}$ of Whitney metrics
($\alpha_{2,0},\alpha_{1,1}$ and $\alpha_{0,2}$ have been 
already defined in \cite{HHNUY}).
When the metric is induced from a cross cap
expressed by the canonical coordinate, then
$a_{2,0}$, $a_{1,1}$ and $a_{0,2}$ 
as in \eqref{eq:fn} coincide with
$\alpha_{2,0},\alpha_{1,1}$ and $\alpha_{0,2}$
for the induced Whitney's metric.

\medskip

Let $d\sigma^2$ be a Whitney metric
defined on a $2$-manifold $M^2$, and $p\in M^2$
a singular point of the metric.
Applying Theorem \ref{thm:main} for $b=0$,
there exists a $C^\infty$ map germ $f$
into $\R^3$ defined on a neighborhood $U$ of $p$
having a cross cap singularity
at $p$ satisfying the following two properties:
\begin{enumerate}
 \item the first fundamental form $d\sigma^2_f$ of $f$
       is formally isometric (cf.\ Definition \ref{def:inf})
       to $d\sigma^2$ at $p$, 
 \item the characteristic function of $f$ 
       is a flat function at $p$, that is, the Taylor expansion at
       $p$ is the zero power series.
\end{enumerate}
If $f$ is real analytic, it is a normal cross cap 
(cf.\ Definition \ref{def:normal}).
However, we do not assume here the real analyticity of 
$d\sigma^2_f$ and $f$.
Taking the normal form of $f$, we may assume 
that $f$ is expressed as
\[
  f(x,y)=
        \left(
	   x,\,\,
	   xy,\,\,
	   z(x,y)\right),
\]
where
\[
   x=x(u,v),\qquad y=y(u,v),\qquad z=z(x,y)
\]
are smooth functions defined on a neighborhood
of $p=(0,0)$.
For each pair of integers $(i,j)$ satisfying
$i+j\ge 2$ and $i,j\ge 0$,
there exists a unique assignment
\[
   d\sigma^2 \mapsto \alpha^{d\sigma^2}_{i,j}\in \R
\]
such that
\[
   [z]=
   \sum_{n=2}^\infty \sum_{i=0}^n
      \frac{\alpha^{d\sigma^2}_{i,n-i}}{i!(n-i)!} x^i y^{n-i}.
\]
So we may regard the series 
\[
 \alpha(d\sigma^2,p):=\{\alpha^{d\sigma^2}_{i,j}\}_{i+j\ge 2,\,\, i,j\ge 0}
\]
as a family of invariants of $d\sigma^2$.
By Theorem \ref{thm:main},
we get the following assertion:

\begin{theorem}
 Let $d\sigma^2_{1}$ and $d\sigma^2_{2}$
 be Whitney metrics on $M^2$
 having a singularity at the same point
 $p\in M^2$.
 Then the two metrics are formally isometric
 if and only if
 $\alpha(d\sigma^2_{1},p)
 =\alpha(d\sigma^2_{2},p)$.
\end{theorem}

In other words, $\alpha$ is a family of complete invariants
distinguishing the formal isometry classes
of Whitney metrics at $p$.
This family of invariants also induces a family of
intrinsic invariants for cross caps in an
arbitrarily given  Riemannian
3-manifold $(N^3,g)$ as follows.
Let $f:M^2\to N^3$ be a $C^\infty$ map which admits 
only cross cap singularities.
Then the induced metric $d\sigma^2_f$ gives
a Whitney metric.
Let $p\in M^2$ be a cross cap singularity of $f$.
Then we set 
\[
  A(f,p):=\alpha(d\sigma^2_f,p),
\]
which can be considered as 
a family of intrinsic invariants
of a germs of cross cap singularities.
When $(N^3,g)$ is the Euclidean $3$-space, 
we can give an explicit algorithm to compute
the invariants as follows:
\begingroup
\renewcommand{\theenumi}{\arabic{enumi}}
\renewcommand{\labelenumi}{\arabic{enumi}.}
\begin{enumerate}
 \item Take the $(m+1)$-st ($m\ge 2$) 
       canonical coordinate system $(u,v)$ centered at $p$,
       that is, $f$ has the following Taylor expansion at $p=(0,0)$:
       \[
         [f]=\left(u, uv+\sum_{n=3}^{m+1} \frac{b_nv^n}{n!} ,
              \sum_{n=2}^{m+1} \sum_{i=0}^n
              \frac{a_{i,n-i}}{i!(n-i)!} u^i v^{n-i}\right)
              +O_{m+2}(u,v).
       \]
       Such a coordinate system can be taken using Fukui-Hasegawa's
       algorithm given in \cite{FH}.
 \item Using this coordinate system $(u,v)$,
       we can determine the coefficients of the following expansion
       up to $(m+1)$-st order terms
       because of the expression $f=(u,0,0)+O_2(u,v)$:
 \begin{align*}
  [E]&=\sum_{i+j\le m+1}\frac{E(i,j)}{i!j!}u^iv^j
  +O_{m+2}(u,v),\\
  [F]&=\sum_{i+j\le m+1}\frac{F(i,j)}{i!j!}u^iv^j
  +O_{m+2}(u,v),\\
  [G]&=\sum_{i+j\le m+1}\frac{G(i,j)}{i!j!}u^iv^j
  +O_{m+2}(u,v),
 \end{align*}
       where $d\sigma^2_f=E\,du^2+2F\,du\,dv+G\,dv^2$.
 \item Setting $b=0$, we compute 
       \begin{align*}
	&X(k,l)\qquad (0\le k+l\le  m+2), \\
	&Y(k,l)\qquad (0\le  k+l\le  m), \\
	&Z(k,l)\qquad (0\le  k+l\le m+1),
       \end{align*}
       according to the algorithm given in the proof of
       Theorem \ref{thm:main}.
 \item We formally set
       \begin{align*}
	u&:=\sum_{i+j\le m}\frac{U(i,j)}{i!j!}x^i y^j+O_{m+1}(x,y),\\
	v&:=\sum_{i+j\le m}\frac{V(i,j)}{i!j!}x^i y^j+O_{m+1}(x,y),
       \end{align*}
       and substitute them into the 
       expansions
 \begin{align*}
  x&=\sum_{i+j\le m}\frac{X(i,j)}{i!j!}u^i v^j+O_{m+1}(u,v),\\
  y&=\sum_{i+j\le m}\frac{Y(i,j)}{i!j!}u^i v^j+O_{m+1}(u,v).
 \end{align*}
       Then we can determine all of the coefficients
       \[
         U(k,l),\quad V(k,l)\qquad (0\le  k+l\le  m).
       \]
 \item Using them, we can finally determine all
       of the coefficients of the expansion
       \begin{equation}\label{eq:finZ}
	[Z]=\sum_{i+j\le m}\frac{A_{i,j}}{i!j!}x^iy^j
	 +O_{m+1}(x,y),
       \end{equation}
       where $\{A_{i,j}\}_{i+j\ge 2}=A(f,p)$.
\end{enumerate}
\endgroup%
However, the 
uniqueness of the expression
\eqref{eq:finZ} was already shown, and
one can alternatively compute 
$\{A_{i,j}\}_{i+j\le m}$
via any suitable method.
We remark that the normal cross cap
shown in the right-hand side of Figure \ref{fig:example}
is drawn using the invariants
$A_{i,j}$ for $0\le i+j\le 11$.

One can get the following tables
of intrinsic invariants;
\begin{align*}
 &A_{2,0}=a_{2,0},
 \qquad A_{1,1}=a_{1,1},\qquad A_{0,2}=a_{0,2}, \\
 & A_{3,0}=
 -\frac{b_3 a_{1,1}^2 a_{2,0}+b_3 a_{2,0}-2 a_{3,0} a_{0,2}^2}{2 a_{0,2}^2},\quad
 A_{2,1}=-\frac{b_3 a_{1,1} a_{2,0}-6 a_{0,2} a_{2,1}}{6 a_{0,2}},\\
 &
 A_{1,2}=-\frac{-b_3 a_{1,1}^2-2 a_{0,2} a_{1,2}-b_3}{2 a_{0,2}},\quad
 A_{0,3}=
 \frac{3 b_3 a_{1,1}+2 a_{0,3}}{2}.
\end{align*}
The numerators of the above invariants have been
computed in \cite{HHNUY}.
The authors also computed
the fourth order invariants $A_{i,j}$ ($i+j=4$), 
which are more complicated.
For example, $A_{0,4}$ has the simplest expression amongst them, which is
given by
{\small 
\[
A_{0,4}=
\frac{4 a_{0,2} \left(4 b_4 a_{1,1}+3 a_{0,4}\right)+3 b_3 \left(b_3 \left(15 a_{1,1}^2-4 a_{0,2} a_{2,0}+7\right)+4 \left(a_{0,3} a_{1,1}+4 a_{0,2} a_{1,2}\right)\right)}{12 a_{0,2}}.
\]}

\begin{ack}
The authors thank Toshizumi Fukui
for his valuable comments. 
\end{ack}


\end{document}